\documentclass[11pt]{amsart}
\usepackage[margin=1in]{geometry}

\usepackage{amsthm,amsmath,amssymb,color,bbm,url,graphicx, color, enumerate}
\usepackage[hidelinks]{hyperref}

\theoremstyle{definition}
\newtheorem{defi}{Definition}[section]

\theoremstyle{plain}
\newtheorem{thm}[defi]{Theorem}
\newtheorem{lemma}[defi]{Lemma}

\newtheorem{prop}[defi]{Proposition}

\newcommand{\set}[1]{\left\{#1\right\}}                   																	%Set
\newcommand{\T}{\ensuremath{{\mathcal T}}}
\newcommand{\gen}[1]{\langle #1\rangle}
\DeclareMathOperator{\stab}{Stab}

\makeatletter
\def\imod#1{\allowbreak\mkern10mu({\operator@font mod}\,\,#1)} %modulo symbol
\def\@setcopyright{}                                           %to get rid of AMS junk at top and bottom of page 
\def\serieslogo@{}
\makeatother

\begin{document}
	\author[M.J.C.~Loquias]{Manuel Joseph C.~Loquias}
	\address[M.J.C.~Loquias]{Institute of Mathematics, University of the Philippines Diliman, 1101 Quezon City, Philippines}
	\email{mjcloquias@math.upd.edu.ph}
	
	\author[R.B.~Santos]{Rovin B.~Santos}
	\address[R.B.~Santos]{Institute of Mathematics, University of the Philippines Diliman, 1101 Quezon City, Philippines}
	\email{rbsantos8@up.edu.ph}

	\title{Perfect precise colorings of plane semiregular tilings}
	
	\begin{abstract}{
		A coloring of a planar semiregular tiling $\T$ is an assignment of a unique color to each tile of $\mathcal T$. 
		If $G$ is the symmetry group of $\T$, we say that the coloring is perfect if every element of $G$ induces a permutation on the finite set of colors. 
		If $\T$ is $k$-valent, then a coloring of $\T$ with $k$ colors is said to be precise if no two tiles of $\T$ sharing the same vertex have the same color. 
		In this work, we obtain perfect precise colorings of some families of $k$-valent semiregular tilings in the plane, where $k\leq 6$.
	}\end{abstract}
	
	
	
	\date{July 31, 2023}
	
	\maketitle

	Aside from their aesthetic and mathematical appeal, symmetrically colored tilings have been studied because of their applications in 
	crystallography and materials science (see~\cite{sen} and references therein).
	Of particular interest are perfect colorings of tilings, that is, 
	colorings where every symmetry of the uncolored tiling sends all tiles of a given color to tiles of the same color~\cite{fr,JW19}.
	In addition, colorings of patterns in the hyperbolic plane have garnered attention because of their connection 
	with quasicrystals and structural chemistry~\cite{de,bpf,miro,hyde}.
	
	The term ``precise coloring'' was coined in \cite{ri2} to refer to a 
	coloring of the regular triangular tiling $(3^n)$ using $n$ colors in which no two tiles of the same color share a common vertex. 
	
	This paper is a continuation of work on identifying perfect precise colorings of planar tilings $\T$ in \cite{ri1, cr, y1, rov}.
	In Section~\ref{methods}, we state general results on perfect precise colorings of $\T$. 
	Theorem~\ref{gencase} gives a characterization of perfect or chirally perfect colorings as partitions of $\T$ that are unions of images of certain tile orbits.
	We demonstrate in Section~\ref{example} how to obtain perfect precise colorings (with $k$ colors) of some families of plane semiregular $k$-valent tilings where $k\leq 6$. 
	These colorings were generated using a combinatorial approach and are verified to be perfect or chirally perfect using Theorem~\ref{gencase}.
	Most of the Archimedean tilings considered in \cite{cr} fall under the more general class of semiregular tilings examined in this work, 
	and his findings become special cases of results in Section~\ref{threevalent}, and Propositions~\ref{resultspqpq}, \ref{p4q4results}, and \ref{33p3presults}. 
	
	In Section~\ref{finalsec}, we extend the definition of precise colorings of $k$-valent semiregular tilings to include those with $n$ colors where $n>k$.
	This generalization was first considered in~\cite{y2}, 
	where chirally perfect colorings of the regular triangular tiling $(3^n)$ using $n+1$ colors in which $n$ different colors appear at each vertex were studied. 
	Perfect precise colorings of a semiregular tiling where the number of colors exceed the valency may be generated using Theorem~\ref{gencase}. 
	One may use GAP \cite{GAP4} to find suitable subgroups of the symmetry group of the uncolored tiling that is needed to obtain such colorings.
	This was what was done to obtain the perfect precise colorings of the $4$-valent semiregular tilings with five colors in Figs.~\ref{3464} and~\ref{6464more}
	and with six colors in Fig.~\ref{6464more}.	
	
	Let us recall that a \emph{planar tiling} refers to a set of tiles $\T=\set{t_i}_{i\in\mathbbm{N}}$ that completely cover the plane without any gaps or overlaps.
	If $H$ is a subgroup of the symmetry group $G$ of a given tiling~$\T$ and $t$ is a tile of $\T$, 
	then we define the \emph{$H$-orbit of $t$} as $Ht=\set{ht\mid h\in H}$.
	Here, we refer to $Ht$ as an $H$-orbit of $\T$. 	
	An edge-to-edge tiling of the plane by regular polygons is called \emph{semiregular} (or \emph{uniform}) if the cyclic arrangement of polygons around all vertices are the same. 
	Every semiregular tiling is denoted based on the sequential arrangement of regular polygons about a vertex. 
	That is, $(p_1.p_2.\cdots.p_k)$ is the semiregular tiling wherein the polygons occuring at any vertex are a $p_1$-gon, a $p_2$-gon$, \ldots,$ and a $p_k$-gon in this order or in reverse. 
	Since the valence of each vertex of the semiregular tiling $(p_1. p_2. \cdots .p_k)$ is $k$, we say that the tiling is \emph{$k$-valent}.  
	Examples of semiregular tilings in the spherical, Euclidean, and hyperbolic planes are shown in Fig.~\ref{fig:semireg}.
	It is well known that there are only a finite number of semiregular tilings in the spherical and Euclidean planes,  
	but an infinite number in the hyperbolic plane.
	Constructing semiregular tilings in the hyperbolic plane was explored in \cite{mi}, and more recently in \cite{datta}. 
	In this work, we deal with particular families of $k$-valent semiregular tilings where $k\leq 6$. 
	
	\begin{figure}[htb]
		\centering
		$\begin{array}{ccc}
			\includegraphics[width=1.4in]{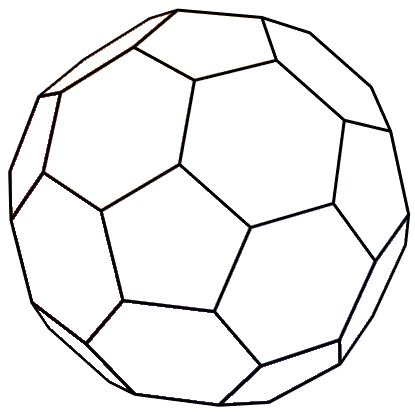} &
			\includegraphics[width=1.4in]{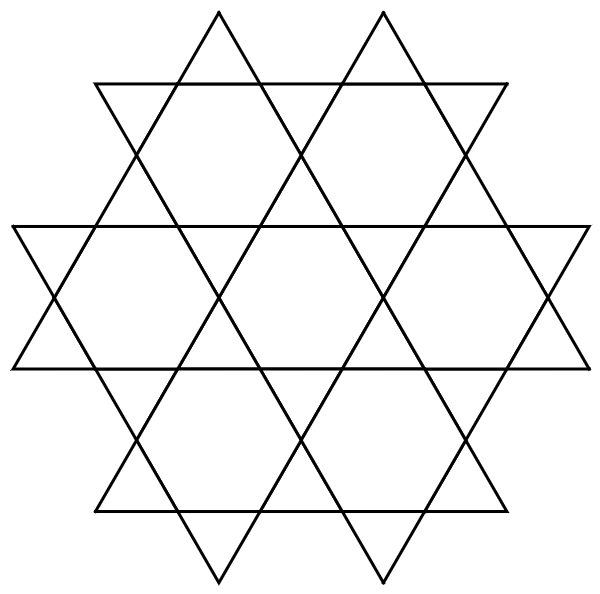} &
			\includegraphics[width=1.4in]{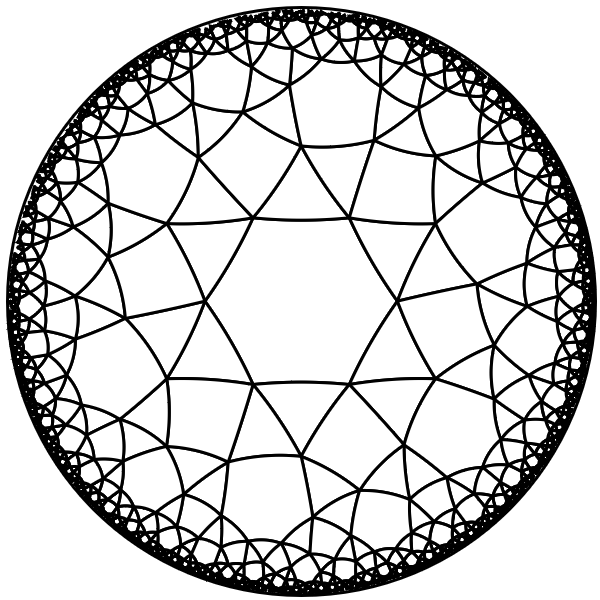}
			\\ (5.6.6) & (3.6.3.6) & (3.3.4.3.6) 
		\end{array}$
		\caption{Some semiregular tilings of the plane.} 
		\label{fig:semireg}
	\end{figure}
	
	An \emph{$n$-coloring} of a semiregular planar tiling $\mathcal T$ is a surjective map from (the tiles of) $\T$ to a finite set of $n$ colors.
	If $\mathcal{T}$ is $k$-valent, then a $k$-coloring of $\T$ is called \emph{precise} if all $k$ colors appear at every vertex of $\T$. 
	A coloring of $\mathcal T$ is said to be \emph{perfect} if every element of the symmetry group~$G$ of~$\mathcal T$ effects a permutation on the set of colors. 
	If $G$ contains indirect symmetries but only the direct symmetries (orientation-preserving symmetries) of $G$ permute the colors in a coloring of $\mathcal T$, 
	then we say that the coloring is \emph{chirally perfect}. 
	A coloring of $\mathcal T$ that has only one orbit of colors under the action of $G$ is called \emph{transitive}.
	The symmetry group of every tiling considered in this study is a triangle group ${\ast}pqr=\gen{P, Q, R: P^2 = Q^2 = R^2 = (QR)^p = (RP)^q = (PQ)^r = 1}$, or a subgroup of this group. 
	
	\section{Methods to Obtain Perfect Precise Colorings}\label{methods}
	
		There are two ways to obtain a perfect precise coloring of a tiling. 
		The first approach is to ensure at the outset that the coloring obtained is precise, 
		while making sure that the colors are permuted by generators of the symmetry group of the tiling.
		This is the approach used in \cite{cr} to obtain all perfect precise colorings of Euclidean semiregular tilings 
		where the number of colors is equal to the valency of the tiling.
		This method entails a combinatorial case-to-case analysis, 
		and becomes quite involved as the valency increases.
		Nevertheless, this method is convenient especially for $3$-valent and $4$-valent tilings.
		
		For the $3$-valent case, the following lemma appears in~\cite{cr}. 
		
		\begin{lemma}\label{3valent}
			Let $\mathcal{T}$ be a 3-valent semiregular planar tiling. 
			If a tile of $\mathcal{T}$ has an odd number of edges, then there is no precise coloring of $\mathcal{T}$.
		\end{lemma}
		
		The proof of Lemma~\ref{3valent} is immediate: if the tile $t$ with an odd number of edges has color 1,
		then the tiles surrounding $t$ need to be colored alternately by colors 2 and 3,
		which cannot happen because of the odd number of edges of $t$.
		
		Let us consider a semiregular planar tiling $\T$ of the form $(p_1.p_2.\cdots.p_k)$ where $k\geq 3$ and the $p_i$'s are distinct.
		Hence, $\T$ is a $k$-valent tiling with $k$ different polygons meeting at every vertex of~$\T$.
		Clearly, assigning a unique color to each of the $k$ types of polygons is the only way to obtain a precise coloring of $\T$.
		Since symmetries of $\T$ are isometries, they can only map a polygonal tile to another tile of the same type.
		This implies that every symmetry of $\T$ fixes each color in the coloring of $\T$, and so the colored tiling is perfect. 
		We write this into a lemma for easy reference in the next section. 
		As far as we can tell, Lemma~\ref{even} does not appear in the literature.
		
		\begin{lemma}\label{even}
			The semiregular planar tiling $(p_1.p_2.\cdots.p_k)$ where $k\geq 3$ and the $p_i$'s are distinct
			has exactly one perfect precise coloring.
		\end{lemma} 
		
		Observe that Lemmas~\ref{3valent} and~\ref{even} imply that there is no semiregular planar tiling $(p.q.r)$
		such that $p$, $q$, and $r$ are distinct and one among them is odd. 
		This result can be verified in~\cite{mi}.		
		
		Another approach is to first obtain perfect colorings of a tiling $\T$, and then discard those that are not precise.
		In this approach, one may use suitable subgroups of the symmetry group of $\T$ to obtain perfect colorings of $\T$.
		A group-theoretic approach to obtain perfect colorings of tilings can be traced back to~\cite{waer,roth}.
		The method described here is based on the work started in~\cite{de}.
		
		If there is only one orbit of tiles under the symmetry group of $\T$, 
		then the following theorem (with $X=\mathcal{T}$ and $H=G$) characterizes perfect colorings of $\T$ as partitions of $\T$ obtained from the orbit of a tile in $\T$ by 
		some subgroup of the symmetry group of $\T$ (cf.~\cite{de,bpf}).  
		The theorem is applicable not only to perfect colorings but also to chirally perfect colorings.
		The statement and proof of the theorem appears in~\cite{Ev}, but for completeness we give here the proof.
		
		\begin{thm}\label{oc}
			Let $X$ be a subset of a planar tiling $\T$ such that the subgroup $H$ of the symmetry group $G$ of $\T$ acts transitively on~$X$.  
			Then the $m$ colors of a coloring of $X$ are permuted by elements of $H$ 
			if and only if the coloring corresponds to a partition of~$X$ of the form 
			\begin{equation}\label{poc}
				\{h (Jt)\mid h\in H\},
			\end{equation}
			where $t\in X$ and $J$ is a subgroup of index $m$ in $H$ that contains $\stab_H(t)$ (the stabilizer of $t$ with respect to $H$).
		\end{thm}
		
		\begin{proof}
			Consider a coloring of $X$ with colors $1,\ldots,m$ such that the colors are permuted by elements of~$H$. 
			Then the group $H$ acts on the set of colors.						 
			Take any tile $t\in X$ and assume that its color is $1$. 
			Let $J$ be the stabilizer of $1$ with respect to $H$.
			Since $H$ acts transitively on $X$, it also acts transitively on the set of $m$ colors.
			It then follows from the Orbit-Stabilizer theorem that $[H:J]=m$.   
			Note that every element of $H$ that fixes $t$ must also fix color $1$, and so $\stab_H(t)\subseteq J$.  
			Finally, the coloring of $X$ corresponds to the partition $\{h (Jt)\mid h\in H\}$ of $X$, 
			where the color of the tiles in $h(Jt)$ is the color of tile $ht$.
			
			Conversely, $\cup_{h\in H}(h(Jt))=X$ since $H$ acts transitively on $X$.  Moreover, because the stabilizer of
			$t$ in $H$ is in $J$, the $m$ sets in $\mathcal{C}=\{h(Jt)\mid h\in H\}$ are pairwise disjoint.  
			It is clear that $\mathcal{C}$ is invariant under $H$, 
			which implies that the elements of $H$ permute the colors in a coloring of $X$ corresponding to $\mathcal{C}$.
		\end{proof}
		
		Observe from the proof of Theorem \ref{oc} that the subgroup $J$ in \eqref{poc} consists of elements of $H$ that fix the color of tile $t$, or equivalently,
		$J$ is the stabilizer of the set $Jt$ under the action of $H$ on the elements of the partition of $X$ in \eqref{poc}. 
		
		We now look at results on how to perfectly color a planar tiling that has more than one transitivity class under the action of its symmetry group.
		While the basic ideas are similar to those discussed in~\cite{JW19} concerning perfect colorings, 
		the succeeding results extend beyond perfect colorings to include chirally perfect colorings as well.
		The next lemma bears resemblance to Lemma 3 of \cite{JW19}. 
		
		\begin{lemma}\label{share}
			Let $\T$ be a planar tiling with symmetry group $G$ and $H$ be a subgroup of $G$. 
			If the $H$-orbits $X_i$ and $X_j$ of $\T$ share some color in a coloring of $\T$ where each element of $H$ permutes the colors, 
			then $X_i$ and $X_j$ have the same set of colors. 
			In particular, $X_i$ and $X_j$ have the same number of colors.
		\end{lemma}
		\begin{proof}
			Suppose that the tiles in $X_i$ have colors $1,2,\ldots,m$, and let $X_i$ and $X_j$ share color $1$. 
			Without loss of generality, assume that tiles $t_i\in X_i$ and $t_j\in X_j$ have color 1.  
			Since the group $H$ acts transitively on $X_i$,
			there exist $h_1, \ldots, h_m \in H$ such that 
			the color of tile $h_kt_i$ is $k$ for $1\leq k\leq m$. 
			Since the colors are permuted by elements of $H$, the color of tile $h_k(t_j)$ should also be $k$ for each $1\leq k\leq m$. 
			Hence all colors that appear in $X_i$ occur in $X_j$.
			The same argument reversed yields the claim.
		\end{proof}
		
		Let $G$ be the symmetry group of a planar tiling $\T$.
		The next lemma (with $O=\T$ and $H=G$) deals with the instance when the same set of colors is used to color each orbit of $\T$ under $G$ 
		to come up with a perfect coloring of $\T$.
		This result aligns closely to Theorem 4 of~\cite{JW19}.
		
		\begin{lemma}\label{norbits}
			Let $O$ be a subset of a planar tiling $\T$ with symmetry group $G$, and $H$ a subgroup of~$G$.
			If all the H-orbits $X_1,\ldots, X_n$ of $O$ have the same set of $m$ colors in a coloring of $O$,
			then the colors of~$O$ are permuted by elements of $H$ if and only if the coloring corresponds to a partition of $O$ of the form
			$\{h(Jt_1\cup\cdots\cup Jt_n)\mid h\in H\}$,
			where $t_i\in X_i$, $1\leq i\leq n$, $J$ is a subgroup of index~$m$ in $H$ that contains $\stab_H(t_1), \ldots,\stab_H(t_n)$.
		\end{lemma}
		\begin{proof}
			It follows from Theorem~\ref{oc} that the coloring of $X_1$ can be viewed as the partition 
			\[\{h(Jt_1)\mid h\in H\}\] for some tile $t_1\in X_1$ with color $c(t_1)$ and $J=\stab_H(c(t_1))$ of index $m$ in $H$ that contains $\stab_H(t_1)$.  
			For $i\in\{2,\ldots,n\}$, choose $t_i\in X_i$ such that tile $t_i$ has color $c(t_1)$.  
			Since elements of $H$ permute the colors in the coloring of $O$, 
			each coloring of $X_i$ corresponds to the partition $\{h(Jt_i)\mid h\in H\}$ and $\stab_H(t_i)$ is contained in $J$ for $2\leq i\leq n$.  
			This shows that the coloring of~$O$ corresponds to the partition $\{h(Jt_1\cup \cdots\cup Jt_n)\mid h\in H\}$ of $O$,
			where the color of the tiles in $h(Jt_1\cup\cdots\cup Jt_n)$ is the color of tile $ht_1$.
			
			For the other direction, it follows from Theorem~\ref{oc} that the $m$ colors in the coloring of each orbit~$X_i$ are permuted by elements of $H$.
			In addition, $J$ is the stabilizer of the color of tiles~$t_1,\ldots,t_n$, that is,
			$J$ is the stabilizer of the subsets $Jt_1,\ldots,Jt_n$ under the action of $H$ on the respective partitions of $X_1,\ldots,X_n$.
			This implies that the action of $H$ on the set of colors in the coloring of each orbit $X_i$ are equivalent.
			Hence, the elements of $H$ permute the colors in the resulting coloring of~$O$.
		\end{proof}
		
		Finally, given a colored tiling $\T$ whose colors are permuted by the subgroup $H$ of the symmetry group $G$ of $\T$,
		we may group the orbits of $\T$ under $H$ having the same set of colors by Lemma~\ref{share}.  
		Upon doing this, we apply Lemma~\ref{norbits} to associate
		the coloring of each collection of orbits as a partition of the union of the orbits.  
		This yields the following general result, which is more straightforward and easier to apply compared to Theorem 5 of \cite{JW19}.
		
		\begin{thm}\label{gencase}
			Let $H$ be a subgroup of the symmetry group $G$ of a planar tiling $\T$, 
			and suppose $X_1,\ldots,X_n$ are the $H$-orbits of $\T$.  
			Then the $m$ colors in a coloring of $\T$ are permuted by elements of $H$
			if and only if it corresponds to a partition of $\T$ of the form $\cup_{i\in I}{O_i}$, where
			\begin{equation}\label{ohi}
				O_i=\bigg\{h\Big(\bigcup_{j=1}^{n_i}J_it_{i,j}\Big)\mid h\in H\bigg\},
			\end{equation}
			with $\sum_{i\in I}{n_i}=n$, $\cup_{i\in I}\cup_{j=1}^{n_i}{t_{i,j}}$ is a complete set of $H$-orbit representatives of $\T$, 
			$J_i$ is a subgroup of $H$ that contains $\stab_H(t_{i,j})$, for $1\leq j\leq n_i$,
			and $m=\sum_{i\in I}[H:J_i]$.
		\end{thm}
		
		Hence, we may use Theorem~\ref{gencase} to obtain all (chirally) perfect colorings of an $m$-valent semiregular planar tiling $\T$.
		To reduce the number of non-precise perfect colorings that will be obtained using this method,
		vertex and edge adjacencies of tiles may be considered a priori.
		For instance, if $k$ tiles in the subset $O_i$ of $\T$ meet at a vertex, 
		then the coloring of $O_i$ must have at least $k$ colors, or equivalently, 
		the subgroup $J_i$ of $H$ must satisfy $[H:J_i]\geq k$.		
		
		It can be shown that this procedure yields the same results in~\cite{cr} for (chirally) perfect colorings of semiregular Euclidean planar tilings.
		Note though that in this approach, knowledge of the subgroup structure of the symmetry groups of the tilings plays an important role.
		This makes the method applicable for specific tilings, but difficult to use when obtaining results for a family of tilings.
			
		To illustrate, consider the $(3.4.6.4)$ tiling $\T$ of the Euclidean plane with symmetry group $G$ of type $\ast 632$.
		The tiling $\T$ has three orbits under the action of $G$: the orbit $X_1$ of hexagons, the orbit~$X_2$ of squares, and the orbit $X_3$ of triangles.
		It has only one perfect precise coloring with four colors.
		This coloring consists of two orbits of colors: the orbit $O_1=X_1\cup X_2$, where $X_1$ and $X_2$ share the same set of three colors, 
		while each tile in the orbit $O_2=X_3$ is assigned the same color (see Fig.~3 of \cite{cr}). 
		The coloring of $O_1$ corresponds to the partition $\{g(Jt_1\cup Jt_2)\mid g\in G\}$, where
		$t_1$ is a hexagon, $t_2$ is a square that is not adjacent to any tile in $Jt_1$, and
		$J$ is a subgroup of index 3 in~$G$ of type $\ast 632$ that contains $\stab_G(t_1)=D_6$ and $\stab_G(t_2)=D_2$.
		On the other hand, the trivial coloring of $O_2$ corresponds to the partition $\{g(Gt_3)\mid g\in G\}=\{X_3\}$, where $t_3$ is any triangle in $\T$.
	
	\section{Examples of Perfect Precise Colorings}\label{example}
	
		In this section, we obtain perfect precise colorings of some semiregular planar tilings $\T$ of valencies~3, 4, 5, and 6.
		We restrict the number of colors in the colorings of $\T$ to be the valency of~$\T$. 
		To obtain these colorings, we start out with a combinatorial argument to exhaust all possible perfect precise colorings. 
		Afterwards, we verify that the obtained colorings of $\T$ are indeed perfect by giving the corresponding partition of $\T$ of the form given in \eqref{ohi}.
		We note that if the coloring of~$\T$ is perfect and there is a vertex $v$ of $\T$ such that no two tiles meeting at $v$ have the same color, then the coloring of $\T$ must be precise.
		This is because the action of the symmetry group of $\T$ is transitive on the set of vertices of $\T$.
	
		\subsection{3-valent Semiregular Planar Tilings}\label{threevalent}
			
			We have the following result: a 3-valent semiregular tiling $\T=(p_1.p_2.p_3)$ admits exactly one perfect precise coloring if and only if each $p_i$ is even. 
			Indeed, the forward direction follows from Lemma~\ref{3valent}.
			For the converse, observe that assigning a tile $t$ by color 1, 
			and then assigning colors 2 and 3 alternately to the surrounding tiles of $t$ gives rise to a precise coloring of $\T$.
			It remains to show that this unique precise coloring of $\T$ is perfect.
			One way to conclude this is to look at the effects of the mirror reflections $P$, $Q$, and $R$ on each of the three colors.
			More formally, we note that the case where $p_1$, $p_2$, and $p_3$ are distinct is covered by Lemma~\ref{even}.
			The case when $\T$ is regular, that is, when $p_1=p_2=p_3=p$ for some $p$, was considered in~\cite{rov}.
			In this case, the symmetry group of $\T$ is $G=\ast p32$ and the precise coloring of $\T$ is perfect because it corresponds to the partition
			$\set{g(Jt)\mid g\in G}$ of $\T$, where $t$ is the $p$-gon stabilized by the $p$-fold rotation $QR$ and $J$ is the subgroup $\gen{PRQRP,Q,R}$ of index $3$ in $G$.
			The last case is, without loss of generality, when $q:=p_1$ is distinct from $2p:=p_2=p_3$ for some integer~$p$.
			Here $G=\ast pq2$ and the precise coloring of $\T$ is perfect because it corresponds to the partition
			$\set{Gt_1}\cup\set{g(Jt_2)\mid g\in G}$ of $\T$, where $t_1$ is any $q$-gon, $t_2$ is the $2p$-gon stabilized by the $p$-fold rotation $QR$, 
			and $J$ is the subgroup $\gen{PRP,Q,R}$ of index 2 in $G$.	
			Fig.~\ref{3colors} illustrates the perfect precise 3-colorings of some 3-valent semiregular tilings,
			where the axes of the generators $PRP$, $Q$, and $R$ of the subgroup $J$ of~$G$ are shown on the coloring of the $(8.6.6)$ tiling.
			
			\begin{figure}[htb]
				\centering 
				$\begin{array}{ccc}
					\includegraphics[height=1.4in]{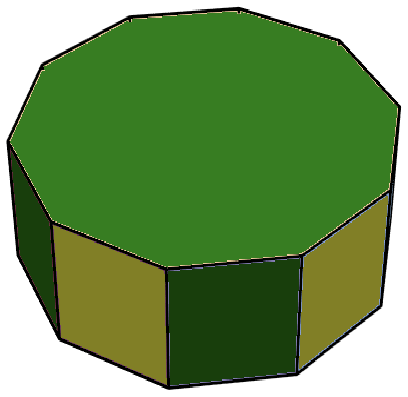} &
					\includegraphics[height=1.4in]{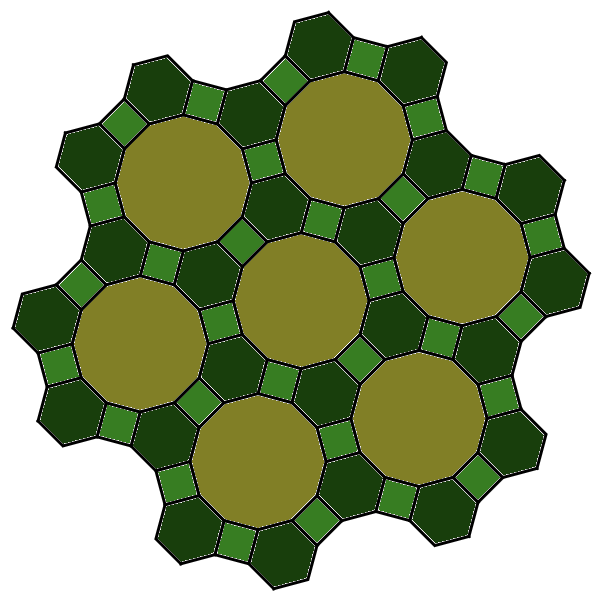} &
					\includegraphics[height=1.4in]{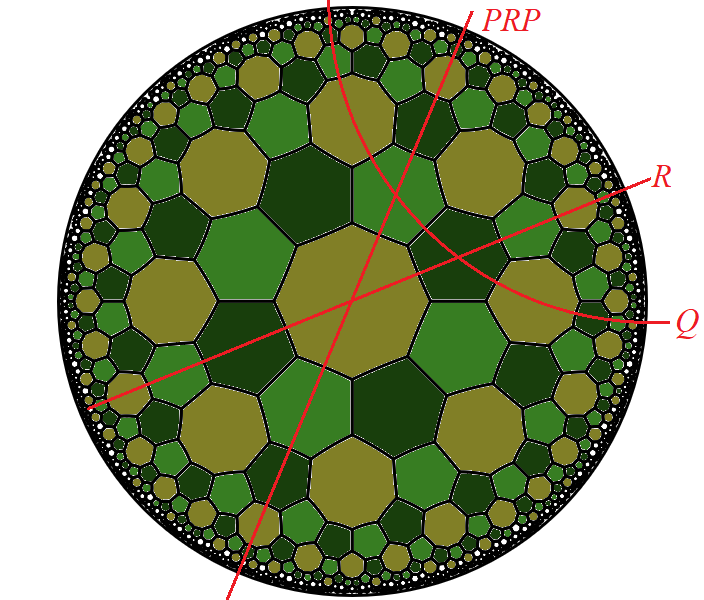} \\ 
					(10.4.4)&(4.6.12)& (8.6.6)\\ 
				\end{array}$
				\caption{Perfect precise 3-coloring of some 3-valent semiregular tilings.}
					\label{3colors}
				\end{figure}
	
		\subsection{4-valent Semiregular Planar Tilings}   
	
			We consider first the family of semiregular planar tilings $\T_1=(p.q.p.q)$. The tiling $\T_1$ consists of two orbits of tiles, the set of $p$-gons and the set of $q$-gons, under its symmetry group $G={\ast}pq2$.  We note the following: (1) the center of each $p$-gon ($q$-gon) is a center of a $p$-fold ($q$-fold) rotation; 
			(2) $p$ $q$-gons are adjacent to every $p$-gon ($q$ $p$-gons are adjacent to every $q$-gon); and 
			(3) $p$ $p$-gons are attached to each $p$-gon at the vertices ($q$ $q$-gons are attached to each $q$-gon at the vertices).  
			Using these properties of $\T_1$, we now outline how a perfect precise 4-coloring of $\T_1$ is achieved.  
			
			Let $t$ be any $p$-gon in the tiling $\T_1$ and say $t$ is given color 1.  
			We then assign colors 2, 3, and 4 to the $p$-gons and $q$-gons that surround $t$.
			There are two ways to assign them to the $p$ $q$-gons that are adjacent to $t$:
			(1) the $q$-gons alternately get two of the three colors, say colors $2$ and $3$ (see Fig.~\ref{fig:f4.5}(a)), or 
			(2) the $q$-gons get colors $2$, $3$, and $4$ in a cyclical manner (see Fig.~\ref{fig:f4.5}(b)). 
			We deal with these two cases separately.
			
			\begin{figure}[htb]
				\centering 
				$\begin{array}{cc}
					\includegraphics[width=1.2in]{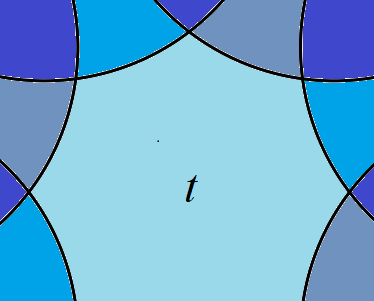}& 
					\includegraphics[width=1.2in]{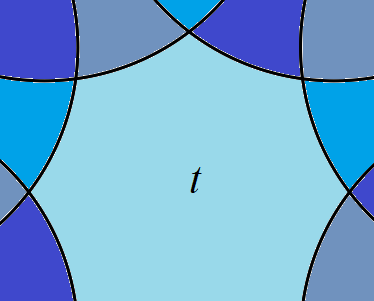} \\ (a) & (b)\\
				\end{array}$
				\caption{Possible color assignments for a perfect precise 4-coloring of the tiling $(p.q.p.q)$. \label{fig:f4.5}} 
			\end{figure}	
			
			If the $p$ $q$-gons adjacent to $t$ are assigned colors $2$ and $3$ alternately, then $p$ must be even for the resulting coloring of $\T_1$ to be precise.
			This assignment forces all $p$ {$p$-gons} attached to $t$ to be assigned color $4$.
			Let us now look at the colored $q$-gons adjacent to~$t$.
			Observe that each of these $q$-gons has a $q$-fold rotational symmetry that sends $t$ (colored~1) to one of the $p$-gons colored $4$ and sends a $p$-gon colored $4$ to $t$.  
			Since we aim for the coloring of $\T_1$ to be perfect, the $q$ $p$-gons attached to each colored $q$-gon must have alternating colors $1$ and $4$.
			To ensure that the coloring of $\T_1$ remains precise, $q$~must also be even.
			In addition the $q$ $q$-gons attached to a $q$-gon colored $2$ should be assigned color~$3$, and 
			all $q$ $q$-gons attached to a $q$-gon colored $3$ should be assigned color $2$.
			Applying the same arguments successively to each layer of $p$-gons and $q$-gons surrounding $t$, 
			we obtain a perfect precise $4$-coloring of $\T_1$ where the $q$-gons adjacent to each $p$-gon are alternately colored $2$ and $3$
			and the $p$-gons adjacent to each $q$-gon are alternately colored $1$ and $4$.    
			The coloring consists of two color orbits and is precise by construction.
			The coloring is indeed perfect because it corresponds to the partition $O_1\cup O_2$ of $\T_1$,
			with $O_1=\set{g(J_1t_1)\mid g\in G}$ where $t_1$ is the $p$-gon stabilized by the $p$-fold rotation $QR$ and $J_1=\gen{PRP,Q,R}$ of index $2$ in $G$ and 
			$O_2=\set{g(J_2t_2)\mid g\in G}$ where $t_2$ is the $q$-gon stabilized by the $q$-fold rotation $PR$ and $J_2=\gen{P,QRQ,R}$ of index $2$ in $G$.
			
			On the other hand, if the $p$ $q$-gons that are adjacent to $t$ are colored $2$, $3$, and $4$ in a repeating manner, 
			then $p$ must be a multiple of three for the resulting coloring of $\T_1$ to be precise.
			This color assignment implies that the $p$ $p$-gons attached to $t$ should also be colored $2$, $3$, and $4$ in a repeating manner such that
			the $p$-gon adjacent to $q$-gons colored $2$ and $3$ (resp.~$2$ and $4$, and $3$ and $4$) should be colored $4$ (resp.~$3$ and $2$). 
			We now turn our attention to the colored $q$-gons attached to~$t$.
			If the $q$-gon $t'$ is colored $2$, then it will be surrounded by three adjacent $p$-gons with colors $4$, $1$, and $3$.
			For the coloring of $\T_1$ to be perfect, the $q$ $p$-gons adjacent to $t'$ need to be colored $1$, $3$, and $4$ in a repeating manner.
			This means that $q$ must also be a multiple of $3$.
			At the same time, the $q$ $q$-gons attached to $t'$ should be colored $1$, $3$, and $4$ in a repeating manner so that the coloring of $\T_1$ remains precise.
			Similar arguments yield the colors of every $p$-gon and $q$-gon surrounding each colored $q$-gon adjacent to $t$.  
			Repeating the same reasoning to each layer of $p$-gons and $q$-gons surrounding $t$ results to a perfect precise $4$-coloring of $\T_1$
			where the $p$-gons and $q$-gons surrounding each $p$-gon and each $q$-gon are colored by the other three colors in a cyclical manner.
			The symmetry group $G$ of $\T_1$ acts transitively on the four colors, and the coloring of $\T_1$ is precise by construction. 
			The coloring of $\T_1$ is also perfect because it corresponds to the partition $\set{g(Jt_1\cup Jt_2)\mid g\in G}$ 
			where $t_1$ is the $p$-gon stabilized by the $p$-fold rotation $QR$, $t_2$ is the $q$-gon stabilized by the reflection $PRQRQRP$,
			and $J=\gen{PRPRPR,PRQRQRP,Q,R}$ of index $4$ in $G$.
			
			Thus, we obtain the following result.
			
			\begin{prop}\label{resultspqpq} 
				Let $\T$ be the semiregular planar tiling $(p.q.p.q)$.
				Then the number of perfect precise $4$-colorings of $\T$ is
				\begin{enumerate}
					\item one if and only if one of the following hold: $p$ and $q$ are even and $p$ or $q$ is not a multiple of three, or
					$p$ and $q$ are multiples of 3 and $p$ or $q$ is not even; and
					
					\item two if and only if $p$ and $q$ are divisible by 6.
				\end{enumerate}
			\end{prop}
			
			Proposition~\ref{resultspqpq} remains valid even if $p=q$ 
			(in~\cite{rov}, the unique perfect precise $4$-colorings of the octahedron $(3^4)$ and the regular Euclidean tiling $(4^4)$ are discussed,	
			and the two perfect precise $4$-colorings of the regular hyperbolic tiling $(6^4)$ are shown). 
			The existence of the unique perfect precise 4-coloring of the Euclidean semiregular tiling $(6.3.6.3)$ in \cite{cr} also agrees with Proposition~\ref{resultspqpq}. 
			Note that Proposition~\ref{resultspqpq} implies that there are no perfect precise 4-colorings of the semiregular spherical tilings $(3.4.3.4)$ and $(3.5.3.5)$.
			We show in Fig.~\ref{fig:f4.6} perfect precise 4-colorings of the semiregular tiling $(p.q.p.q)$ for some values of $p$ and~$q$.
			The axes of the generators of the subgroups $J_1=\gen{PRP,Q,R}$ and $J_2=\gen{P,QRQ,R}$ of $G$ are drawn on the coloring of the $(6.4.6.4)$ tiling.
			
			\begin{figure}[htb]
				\centering
				$\begin{array}{ccc}
					\includegraphics[height=1.4in]{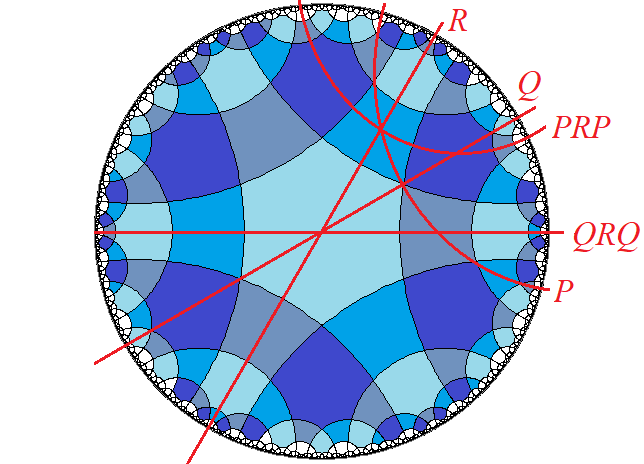} &
					\includegraphics[height=1.4in]{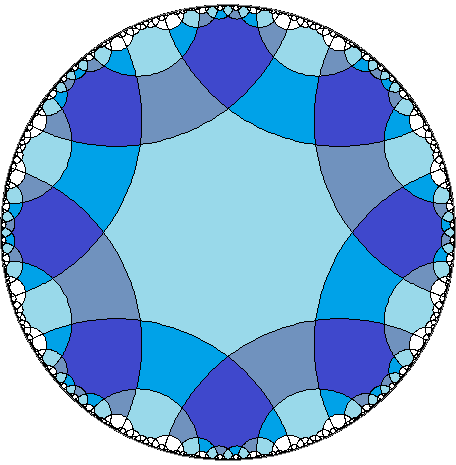} &
					\includegraphics[height=1.4in]{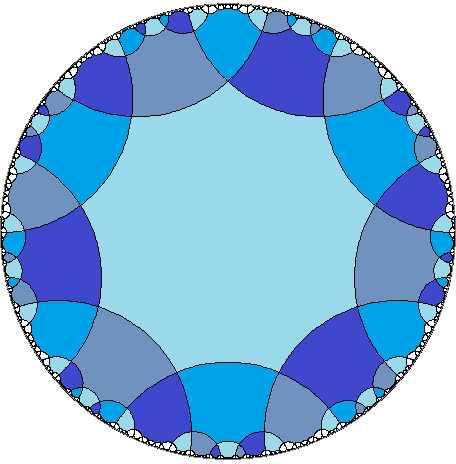}\\
					(6.4.6.4) & (8.4.8.4) & (9.6.9.6)\\
					\includegraphics[height=1.4in]{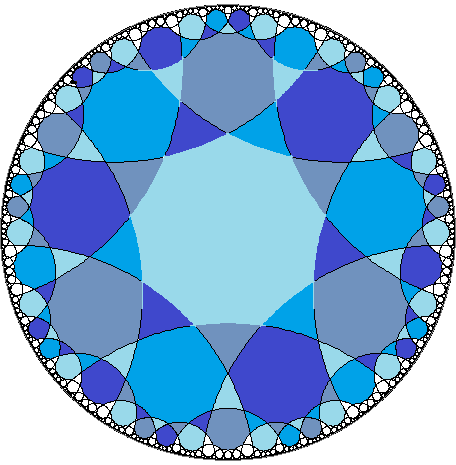} &
					\includegraphics[height=1.4in]{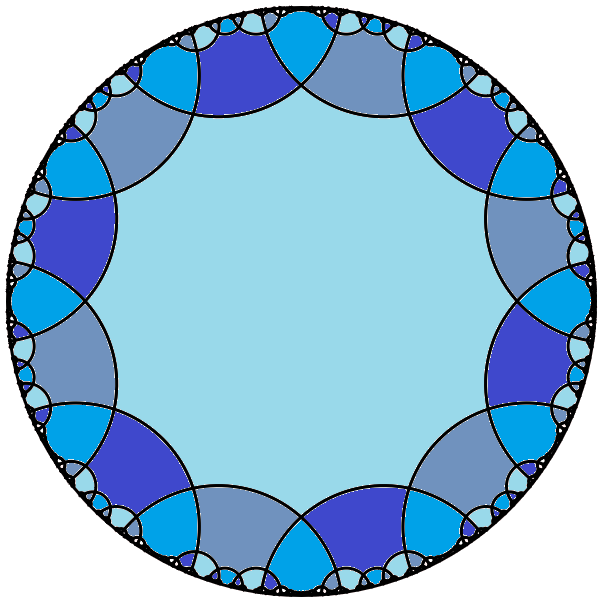}&
					\includegraphics[height=1.4in]{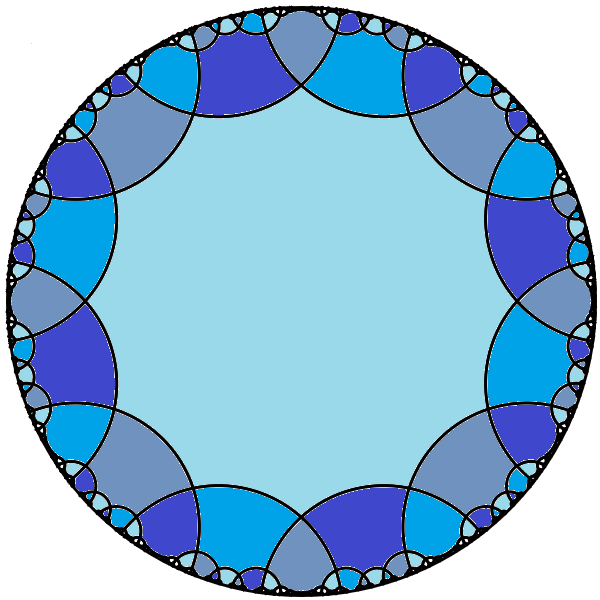} \\
					(9.3.9.3) & (6.12.6.12)_1 & (6.12.6.12)_2
				\end{array}$
				\caption{Some perfect precise 4-colorings of the tiling $(p.q.p.q)$.}
				\label{fig:f4.6}
			\end{figure} 
			
			Another semiregular tiling with symmetry group $G={\ast}pq2$ is the tiling $\T_2=(p.4.q.4)$, where $p\neq q$.  
			There is a $p$-fold ($q$-fold) rotation symmetry of $\T_2$ about the center of each $p$-gon ($q$-gon),
			and a $2$-fold rotation symmetry about the center of every square.
			The action of $G$ on $\T_2$ forms at most 3 orbits of tiles: the set of squares, the set of $p$-gons, and the set of $q$-gons.  
			No two $p$-gons ($q$-gons) in $\T_2$ share a common vertex nor a common edge, and  
			each $p$-gon ($q$-gon) is surrounded by $p$ ($q$) adjacent squares.  
			
			We use a similar approach as for the tiling $\T_1$ to obtain perfect precise 4-colorings of the tiling~$\T_2$. 
			We start with any $p$-gon $t$ in $\T_2$ and assign it the color $1$.
			There are two ways to assign colors to the $p$ squares adjacent to $t$: 
			alternating using two colors, say $2$ and $3$,
			or cyclically using the three colors~$2$, $3$, and $4$.
			The color assignment of the squares will force the color assignment of the $p$ $q$-gons around $t$.
			Afterwards, we color the $q$ $p$-gons and $q$ squares surrounding each $q$-gon around $t$.
			If the $q$-gon $t'$ attached to $t$ has color $2$ (resp.~$3$ and $4$), then there are two possibilities: 
			assign color~$1$ to every $p$-gon attached to $t'$ and assign colors $3$ and $4$ (resp.~$2$ and $4$, and $2$ and $3$) alternately to the squares around $t'$,
			or assign colors $1,3,4$ (resp.~$1,2,4$ and $1,2,3$) to the $p$-gons and squares in some repeating manner.
			The color assignment of the succeeding layers of $p$-gons, $q$-gons, and squares around $t$ are then determined inductively.
			We list the possibilities below and summarize them in the proposition thereafter.
			
			\begin{enumerate}[(i)]
				\item If $p$ is even, then the squares around $t$ are colored $2$ and $3$ in succession while every $q$-gon attached to $t$ are colored $4$.
				
				If $q$ is even, then the $p$-gons surrounding each $q$-gon attached to $t$ are colored $1$
				while the surrounding squares are colored $2$ and $3$ in succession (see Fig.~\ref{fig:f4.7}(a)).
				In this case, the perfect precise $4$-coloring of $\T_2$ corresponds to the partition $O_1\cup O_2\cup O_3$,
				with $O_1$ the set of $p$-gons, $O_2$ the set of $q$-gons, and $O_3=\set{g(Jt)\mid g\in G}$ where 
				$t$ is the square stabilized by the $2$-fold rotation $PQ$ and $J=\gen{P,Q,RPR,RQR}$ of index $2$ in $G$.
				
				If $q$ is a multiple of $3$ then the $p$-gons and squares surrounding each $q$-gon attached to~$t$
				are colored $1,2,3$ in some repeating manner (see Fig.~\ref{fig:f4.7}(b)).
				Here, the perfect precise $4$-coloring of $\T_2$ corresponds to the partition $O_1\cup O_2$,
				with $O_1$ the set of $q$-gons and $O_2=\set{g(Jt_1\cup Jt_2)\mid g\in G}$ where
				$t_1$ is the $p$-gon stabilized by the $p$-fold rotation $QR$, $t_2$ is the square stabilized by the reflection $PRPRP$, and $J=\gen{PRPRP,PRQRP,Q,R}$ of index~$3$ in~$G$.
				
				\item If $p$ is a multiple of 3 and $q$ is even, then interchanging the values of $p$ and $q$ yields the case above.
				
				The case where $p$ and $q$ are odd multiples of $3$ does not give rise to a perfect precise $4$-coloring of $\T_2$.
				Indeed, since $p$ is an odd multiple of $3$, then the $p$ squares and $p$ $q$-gons around $t$ can only be colored $2,3,4$ in a cyclical fashion for the coloring to be precise.
				Consider a $q$-gon $t'$ attached to $t$ with color $4$.
				We can only color the $p$-gons and squares around $t'$ by $1,2,3$ in some repeating manner because $q$ is an odd multiple of $3$ (see Fig.~\ref{fig:f4.7}(c)).
				This will mean that the second $p$-gon adjacent to a square colored $2$ will be assigned color $3$.
				However, the same $p$-gon is attached to a $q$-gon colored 3, and thus the coloring obtained will not be precise.
			\end{enumerate}
		
			\begin{figure}[htb]
				\centering $\begin{array}{ccc}
					\includegraphics[width=1.2in]{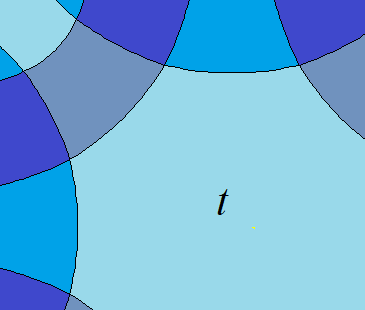}& 
					\includegraphics[width=1.2in]{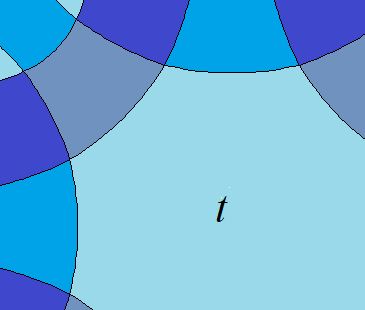}&
					\includegraphics[width=1.2in]{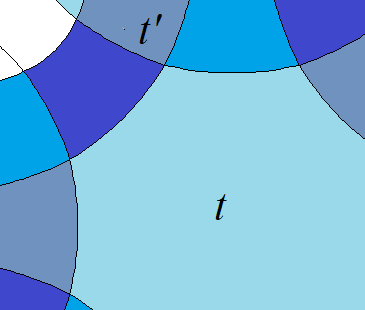}\\ 
					(a) & (b) & (c)  
				\end{array}$
				\caption{Possible color assignments for a perfect precise 4-coloring of the tiling $(p.4.q.4)$.} 
				\label{fig:f4.7}
			\end{figure}
			
			\begin{prop}\label{p4q4results} 
				Let $\T$ be the semiregular tiling $(p.4.q.4)$ with $p\neq q$.
				Then the number of perfect precise $4$-colorings of $\T$ is
				\begin{enumerate}
					\item one if and only if $p$ and $q$ are even and neither is a multiple of~3, or 
					one of $p$ and $q$ is even and the other is an odd multiple of 3;
					
					\item two if and only if $p$ and $q$ are even and exactly one of them is a multiple of 3; and
					
					\item three if and only if $p$ and $q$ are multiples of~6.
				\end{enumerate}
			\end{prop}
			
			Proposition \ref{p4q4results} is consistent with the unique perfect precise 4-coloring of the Archimedean tiling $(3.4.6.4)$ obtained in \cite{cr}. 
			The perfect precise 4-colorings of the semiregular hyperbolic tilings~$(p.4.q.4)$ for some values of $p$ and $q$ in Fig.~\ref{fig:f4.8} illustrates Proposition \ref{p4q4results}.
			The axes of the generators of the corresponding subgroups $J$ of $G$ are shown for the two distinct perfect precise $4$-colorings of the $(8.4.6.4)$ tiling.
			The semiregular spherical tiling $(5.4.3.4)$ admits no perfect precise coloring with 4 colors. 
			
			\begin{figure}[htb]
				\centering
				$\begin{array}{ccc}
					\includegraphics[height=1.5in]{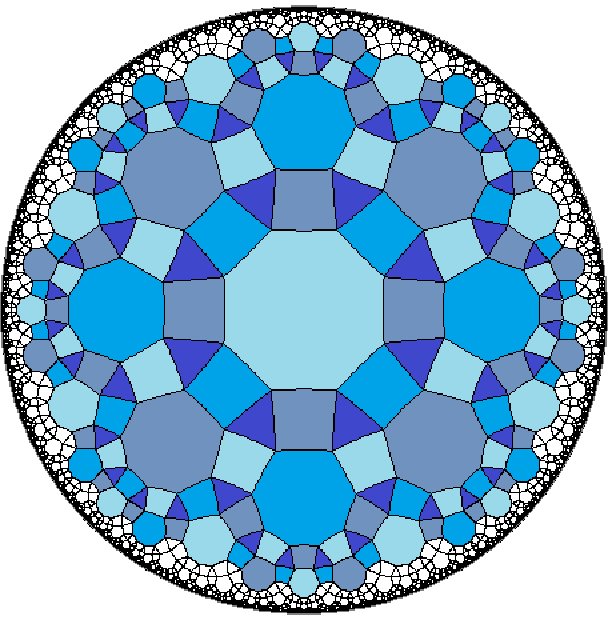}&
					\includegraphics[height=1.5in]{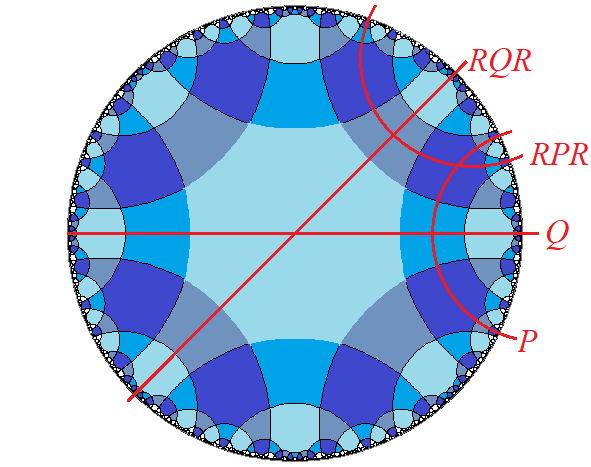}&
					\includegraphics[height=1.5in]{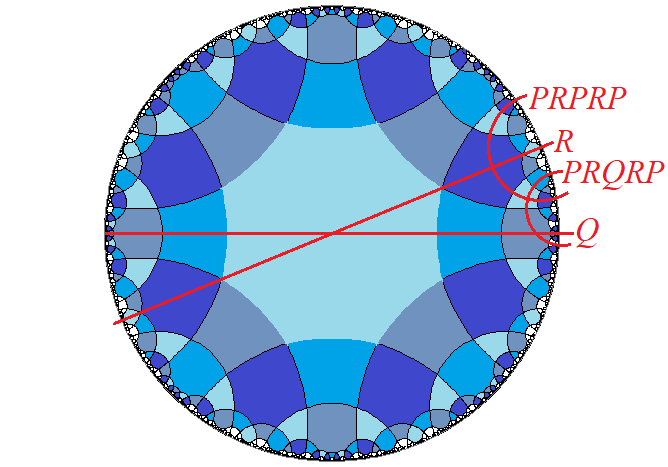}\\
					(8.4.3.4) & (8.4.6.4)_1  &  (8.4.6.4)_2\\
					\includegraphics[height=1.5in]{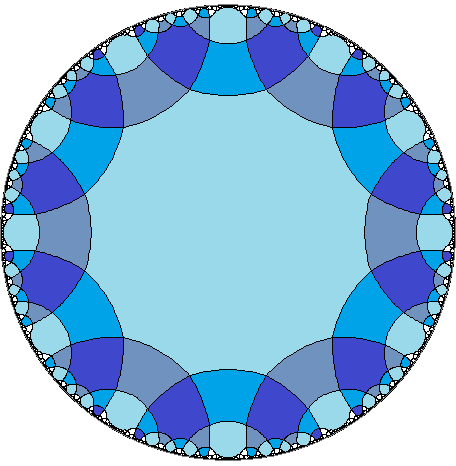}& 
					\includegraphics[height=1.5in]{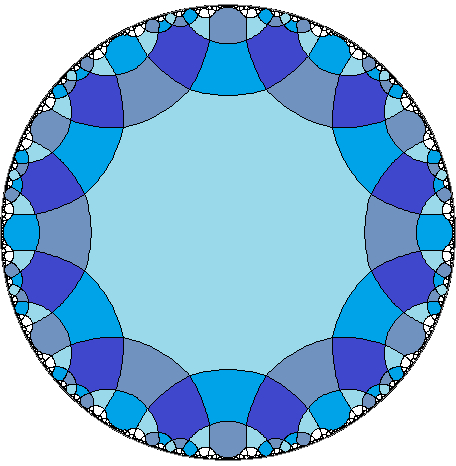}&
					\includegraphics[height=1.5in]{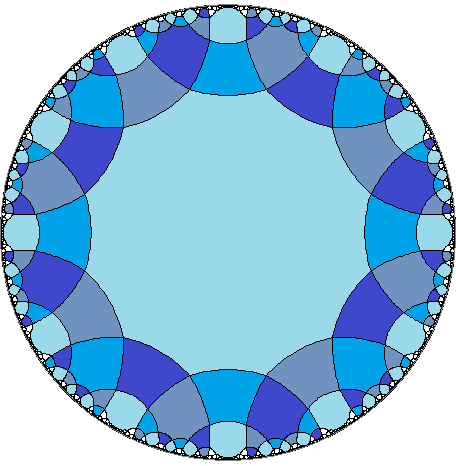}
					\\(12.4.6.4)_1 & (12.4.6.4)_2 & (12.4.6.4)_3
				\end{array}$
				\caption{Some perfect precise 4-colorings of the tiling $(p.4.q.4)$.}
				\label{fig:f4.8}
			\end{figure}
	
		\subsection{5-valent Semiregular Planar Tilings}
	
			We now determine the perfect and chirally perfect precise 5-colorings of the family of 5-valent tilings $\T=(3.3.p.3.q)$.
			If $p\neq q$, then the symmetry group of $\T$ is $G=pq2$. 
			The tiling $\T$ has a $p$-fold ($q$-fold) rotation symmetry about the center of each $p$-gon ($q$-gon),
			and a $2$-fold symmetry at the midpoint of the shared edge of any two adjacent triangles.
			The set of $p$-gons, the set of $q$-gons, and the set of triangles form distinct orbits under the action of $G$ on $\T$.  
			If $p=q$, then the symmetry group of $\T$ is $G'=p{\ast}2$. 
			In this case, $\T$ has a reflection symmetry along the shared edge of each pair of adjacent triangles. 
			The set of $p$-gons form one orbit under the action of $G'$ on $\T$, but is partitioned into two orbits under the action of the rotation group $H=pp2$. 
			It is easy to see that two $p$-gons that have a common vertex belong to different $H$-orbits.
			
			First, we show two things: (a) if $p=q$, then $\T$ does not admit any perfect precise $5$-coloring, and 
			(b) if $p\neq q$, then every $p$-gon is assigned the same color or every $q$-gon is assigned the same color in any perfect precise 5-coloring of $\T$. 
			
			Let $p=q$. Assign distinct colors to five tiles around a vertex of $\T$ as shown in Fig.~\ref{noppc}(a).
			Notice that $\T$ has a reflection symmetry $\alpha$ along the red line that passes through the shared edge of triangles $b$ and~$c$.
			If the $5$-coloring of $\T$ is to be perfect, then $\alpha$ must fix the color of triangle~$a$.
			This means that neither triangle $b$ nor $c$ should be assigned the same color as triangle $a$, otherwise the resulting $5$-coloring of $\T$ will not be precise.
			Hence, we are left with no choice but for the $p$-gon~$d$ to have the same color as triangle $a$.
			This yields a contradiction because $\alpha$ sends the $p$-gon~$d$ to a $p$-gon of a different color.  
			We have now shown that it is impossible to perfectly color~$\T$ with 5 colors such that the resulting coloring is precise whenever $p=q$. 
			
			\begin{figure}[htb]
				\centering
				$\begin{array}{cc}
					\includegraphics[height=1in]{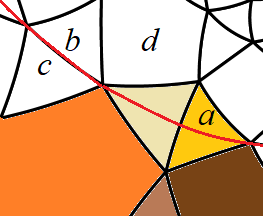}&
					\includegraphics[height=1in]{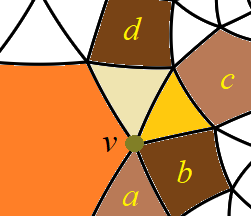}\\
					(a)&(b)\\
				\end{array}$
				\caption{(a) Nonexistence of a perfect precise 5-coloring of the tiling $(3.3.p.3.p)$; (b) Nonexistence of a transitive perfect precise 5-coloring of the tiling $(3.3.p.3.q)$.}
				\label{noppc}
			\end{figure}
			
			We now proceed to show that if $p\neq q$, then all $p$-gons have the same color or all $q$-gons have the same color in a perfect precise $5$-coloring of $\T$.
			
			Just as before, we start with five tiles meeting at a vertex $v$ that are assigned different colors (see Fig.~\ref{noppc}(b)). 
			For the 5-coloring of $\T$ to be precise, the $p$-gon~$c$ can be colored with either the color of triangle $a$ or the color orange of the $p$-gon adjacent to triangle $a$.
			However, the $p$-fold rotation symmetry $\beta$ about the center of the orange $p$-gon fixes the color orange.
			This means that if $p$-gon~$c$ is colored orange, then every $p$-gon in the tiling should be colored orange for the $5$-coloring of $\T$ to be perfect.
			On the other hand, if $p$-gon $c$ has the same color as triangle $a$, then the $q$-gon $d$ will be forced to have the same color brown of the $q$-gon $b$.
			As $\beta$ now fixes color brown, every $q$-gon in $\T$ will be colored brown so that the resulting $5$-coloring of $\T$ is perfect.
			This proves the claim.
			
			We now discuss how to obtain perfect precise $5$-colorings of $\T$ if $p\neq q$.
			We classify the triangles around a $p$-gon ($q$-gon) into two types: 
			those that share a common edge with the $p$-gon ($q$-gon) will be called type I triangles 
			while those that share only a vertex with the $p$-gon ($q$-gon) will be called type II triangles. 
			Suppose $t$ is a $p$-gon and $t'$ is a $q$-gon in $\T$ having a common vertex.		
			We assign different colors to the tiles around this vertex, in particular, we assign color $1$ to the $p$-gon~$t$, 
			color~$2$ to the triangle adjacent to both $t$ and $t'$, color $3$ to the $q$-gon, color $4$ to the triangle adjacent to~$t'$ but not to $t$, 
			and color $5$ to the remaining triangle adjacent to $t$ but not to~$t$'.
			From earlier, we have the following three cases.
			\begin{enumerate}[(i)]
				\item Suppose all $p$-gons are colored $1$ and all $q$-gons are colored $3$. 
				The type I triangles around~$t$ can be colored either $2$ and $5$ alternately or $2,5,4$ in a cyclical manner.
				In the former, $p$ must be even and all type II triangles around $t$ will be colored $4$. 
				This color assignment yields a type I triangle and a type II triangle around $t'$ that are colored $4$. 
				Hence, the type I and type~II triangles around $t'$ must be colored $2$, $4$, $5$ in some repeating manner which means that $q$ must be a multiple of $3$ (see Fig.~\ref{33p3qcases}(a)).
				The colors of the layers of triangles around $t$ are now uniquely determined, and the resulting $5$-coloring of $\T$ is precise by construction.
				In addition, the $5$-coloring is perfect because it corresponds to the partition $O_1\cup O_2\cup O_3$ of $\T$, where
				$O_1$ is the set of $p$-gons, $O_2$ is the set of $q$-gons, and
				$O_3=\set{g(Jt)\mid g\in G}$ where $t$ is any type II triangle around $t$ and $J=\gen{RQ,(PR)^3,PRQRPQ}$ of index $3$ in $G$.
				
				If instead the type I triangles around $t$ are colored $2,5,4$ in a cyclical manner, 
				then similar arguments show that $p$ must be a multiple of three and $q$ must be even for a perfect precise $5$-coloring of $\T$ to be obtained (see Fig.~\ref{33p3qcases}(b)).
				In fact, the resulting coloring is exactly the same coloring above where the roles of $p$ and $q$ are interchanged.
				
				\item Suppose all $p$-gons are colored $1$ but not all $q$-gons are colored $3$ (see Figs.~\ref{33p3qcases}(c) and (d)).
				The $q$-gon adjacent to the triangle colored $5$ should be colored $2$, otherwise, the $5$-coloring of $\T$ will not be precise or all $q$-gons will be colored $3$.
				This implies that the type I triangles and $q$-gons surrounding $t$ can be cyclically colored by the three colors $2,5,3$ or the four colors $2,5,4,3$.
				An analogous argument shows that in the former, $p$ and $q$ must be multiples of $3$ and the resulting perfect precise $5$-coloring of $\T$ 
				corresponds to the partition~$O_1\cup O_2$ of $\T$, with $O_1$ the set of $p$-gons and $O_2=\set{g(Jt_1\cup Jt_2)\mid g\in G}$ where
				$t_1$ is the $q$-gon stabilized by the $q$-fold rotation $PRQRPQRP$, $t_2$ is any type II triangle around $t$, 
				and $J=\gen{RQ,PRQRPQRP,PRQRQRPQ}$ of index $4$ in $G$.				
				
				If instead the type I triangles and $q$-gons around $t$ are assigned four colors,
				we obtain that $p$ should be a multiple of $4$ and $q$ should be even.
				The $5$-coloring of $\T$ obtained is perfect and precise that corresponds to the partition $O_1\cup O_2$,
				with $O_1$ the set of $p$-gons, $O_2=\set{g(Jt_1\cup Jt_2)\mid g\in G}$ where
				$t_1=t'$, $t_2$ is a type I or type II triangle around $t$ with color $3$, 
				and $J=\gen{(RQ)^4,PR,QRPRPQ}$ of index $4$ in~$G$.
				
				\item Suppose all $q$-gons are colored $3$ but not all $p$-gons are colored $1$ (see Figs.~\ref{33p3qcases}(e) and (f)).
				The $p$-gon adjacent to the triangle colored $4$ should be colored $2$, otherwise, all $p$-gons will be colored $1$.
				This implies that the type I triangles and $p$-gons surrounding $t$ can be colored either~$2$ and $5$ alternately or $2,5,4$ in a cyclical manner.
				Similar arguments show that to obtain a perfect precise $5$-coloring of $\T$,
				$p$ must be even and $q$ must be a multiple of $4$ in the former while $p$ and $q$ must be divisible by $3$ in the latter.
				In both instances, the coloring generated is the same coloring obtained from Case (ii) when the values of $p$ and $q$ are interchanged.
			\end{enumerate}
		
			\begin{figure}[htb]
				\centering
				$\begin{array}{ccc}
					\includegraphics[height=1in]{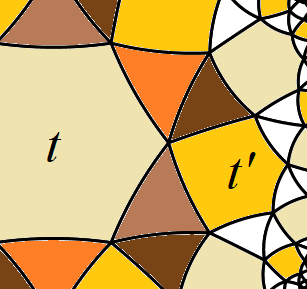}&
					\includegraphics[height=1in]{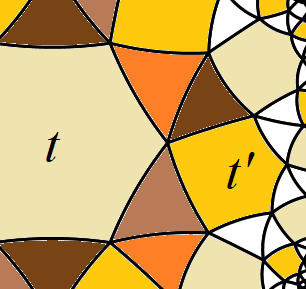}&
					\includegraphics[height=1in]{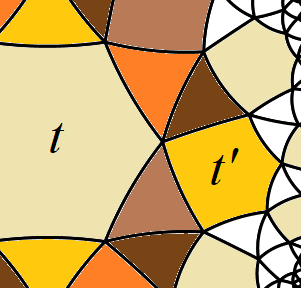}\\
					(a) & (b) & (c) \\
					\includegraphics[height=1in]{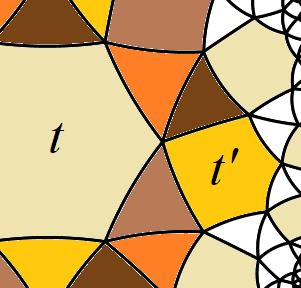}&
					\includegraphics[height=1in]{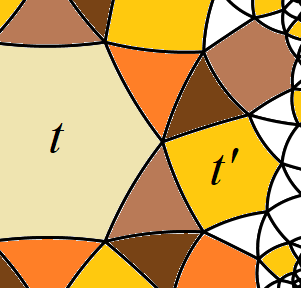}&
					\includegraphics[height=1in]{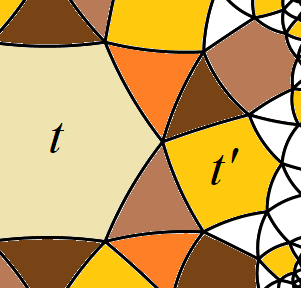}\\
					(d) & (e) & (f) \\
				\end{array}$
				\caption{Possible color assignments for a perfect precise 5-coloring of the tiling $(3.3.p.3.q)$, with $p\neq q$.}
				\label{33p3qcases}
			\end{figure}
			
			The following proposition summarizes the preceding discussion. 
			
			\begin{prop}\label{33p3qresults}
				Let $\T$ be the semiregular tiling $(3.3.p.3.q)$ with $p\neq q$.
				Then the number of perfect precise $5$-colorings of $\T$ is
				\begin{enumerate}
					\item one if and only if one of the following hold: 
					one of $p$ and $q$ is a multiple of $3$ but not of~$4$ and the other is even that is not a multiple of $4$,
					one of $p$ and $q$ is a multiple of $4$ but not of $3$ and the other is even that is not a multiple of $4$, or
					one of $p$ and $q$ is a multiple of $4$ but not of $3$ and the other is an odd multiple of $3$;
					
					\item two if and only if one of the following hold:
					one of $p$ and $q$ is a multiple of $12$ and the other is even that is neither a multiple of $3$ nor of $4$,
					both $p$ and $q$ are odd multiples of $3$,
					both $p$ and $q$ are multiples of $4$ but not of $3$, or
					one of $p$ and $q$ is a multiple of $6$ but not of $4$ and the other is a multiple of $4$ but not of $3$;
					
					\item three if and only if one of the following hold:
					one of $p$ and $q$ is a multiple of $6$ and the other is an odd multiple of $3$, or
					one of $p$ and $q$ is a multiple of $12$ and the other is a multiple of~$4$ but not of $3$;
					
					\item four if and only if both $p$ and $q$ are multiples of $6$ but not of $4$;
					
					\item five if and only if one of $p$ and $q$ is a multiple of $12$ and the other is a multiple of $6$ but not of $4$; and
					
					\item six if and only if both $p$ and $q$ are multiples of $12$.				
				\end{enumerate}
			\end{prop}
			
			The existence of three perfect precise 5-colorings for the tiling $(3.3.3.3.6)$ in \cite{cr} agrees with Proposition~\ref{33p3qresults}.
			Examples of perfect precise 5-colorings of the tiling $(3.3.p.3.q)$ with $p\neq q$ are presented in Fig.~\ref{33p3qexamples}.
			
			\begin{figure}[htb]
				\centering
				$\begin{array}{ccc}
					\includegraphics[width=1.35in]{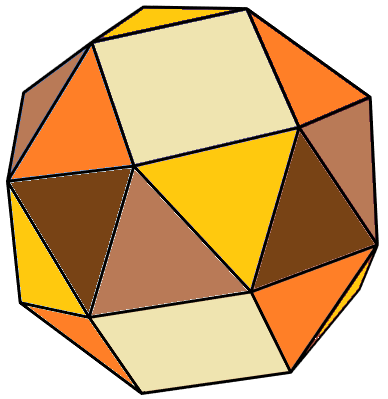}&
					\includegraphics[width=1.4in]{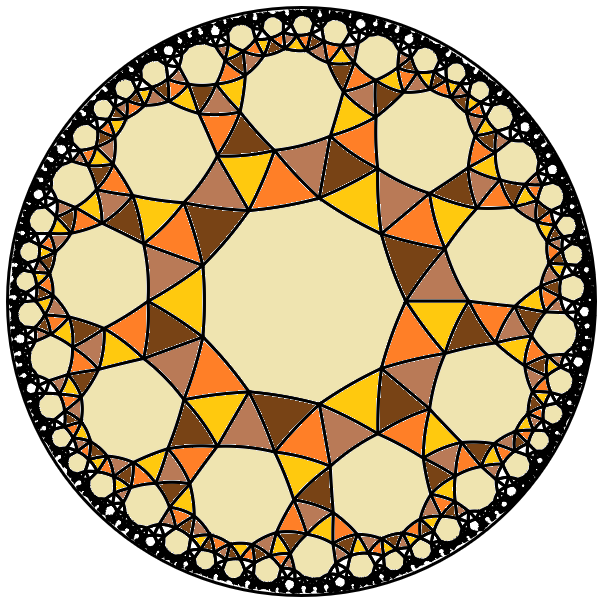}&
					\includegraphics[width=1.4in]{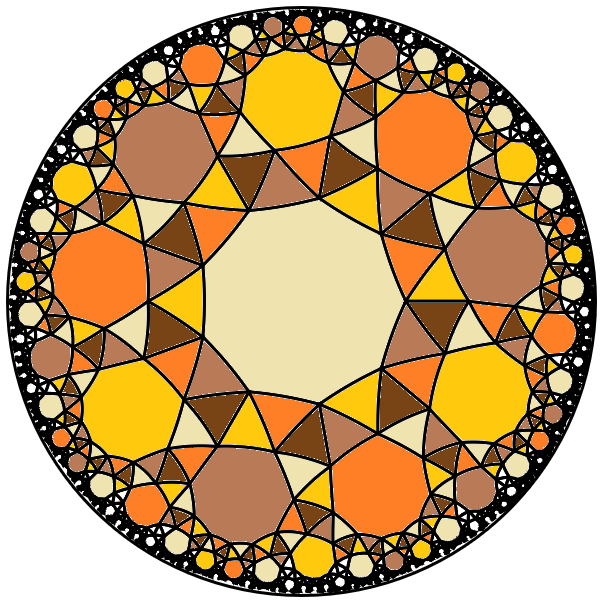}\\
					(3.3.3.3.4)&(3.3.3.3.9)_1&(3.3.3.3.9)_2\\ 
					\includegraphics[width=1.4in]{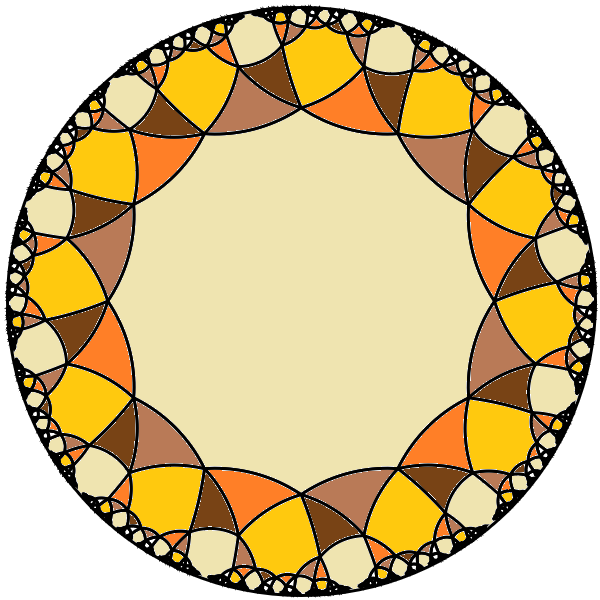}&
					\includegraphics[width=1.4in]{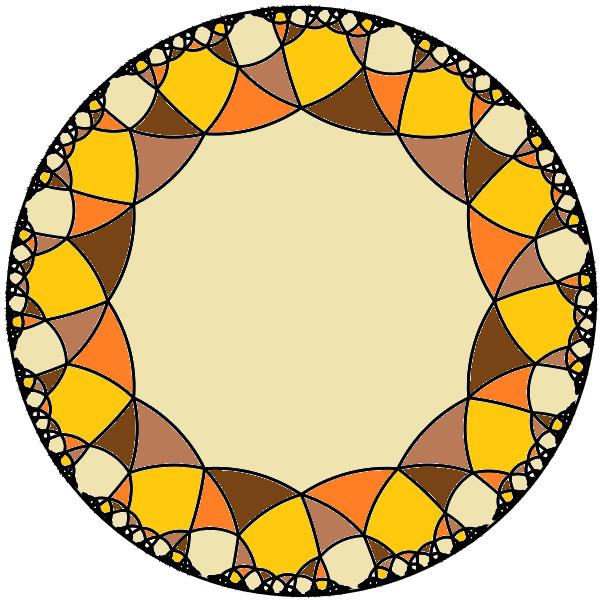}&
					\includegraphics[width=1.4in]{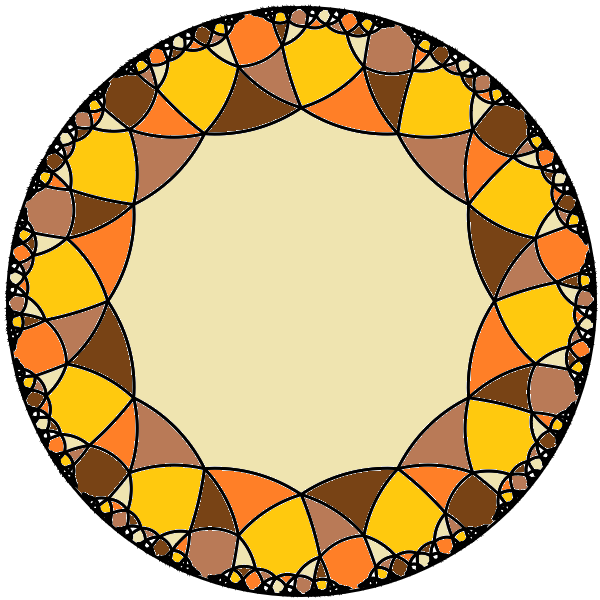}\\
					(3.3.6.3.12)_1&(3.3.6.3.12)_2&(3.3.6.3.12)_3
				\end{array}$
				
				$\begin{array}{cc}
					\includegraphics[width=1.4in]{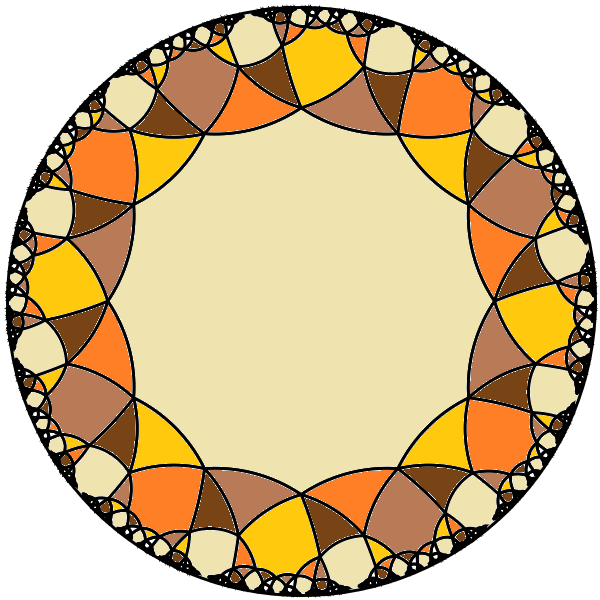}&
					\includegraphics[width=1.4in]{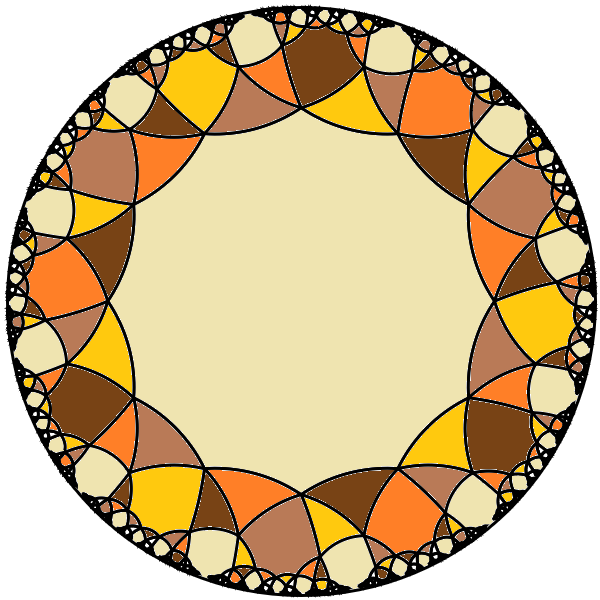}\\
					(3.3.6.3.12)_4&(3.3.6.3.12)_5
				\end{array}$
				\caption{Some perfect precise 5-colorings of the tiling $(3.3.p.3.q)$, with $p \neq q$.}  
				\label{33p3qexamples}
			\end{figure}
			
			Recall that if $p=q$ then the tiling $\T$ has no perfect precise $5$-coloring.
			Nevertheless, if we consider instead the action of the rotation group $H=pp2$ on the tiling $\T$,
			then some of the discussion above is applicable and results to chirally perfect precise 5-colorings of $\T$.
			We state this explicitly in the next proposition.
			
			\begin{prop}\label{33p3presults}
				The number of chirally perfect precise $5$-colorings of the semiregular tiling ${\T=(3.3.p.3.p)}$ is
				\begin{enumerate}
					\item one if and only if $p$ is an odd multiple of $3$ or $p$ is a multiple of $4$ but not of $3$;
					
					\item two if and only if $p$ is a multiple of $6$ but not of $4$; and
					
					\item three if and only if $p$ is a multiple of $12$.				
				\end{enumerate}
			\end{prop}
			
			Note that two chirally perfect precise colorings of the tiling $(3.3.4.3.4)$ using 5 colors were obtained in \cite{cr}, each of which reflects into the other.  
			These two colorings correspond to the same partition of the tiling and in effect, these colorings are the same.
			This is in agreement with Proposition~\ref{33p3presults}. 
			Fig.~\ref{33p3qexampleschiral} shows the two chirally perfect precise 5-colorings of the tiling $(3.3.6.3.6)$. 
			
			\begin{figure}[htb]
				\centering
				$\begin{array}{cc}
					\includegraphics[width=1.4in]{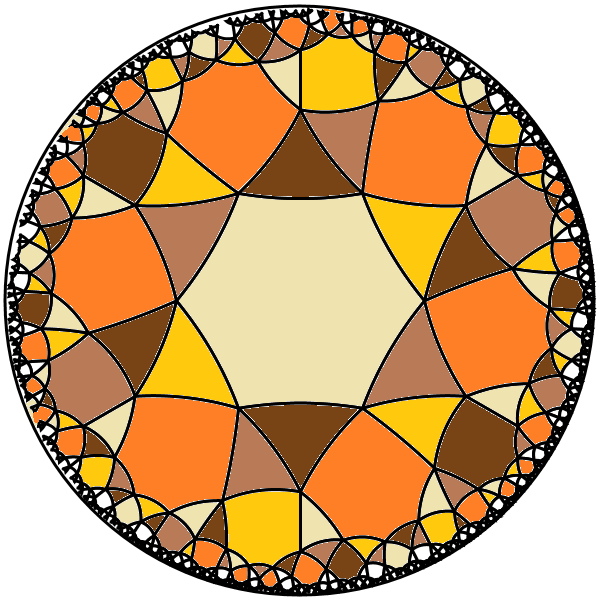}&
					\includegraphics[width=1.4in]{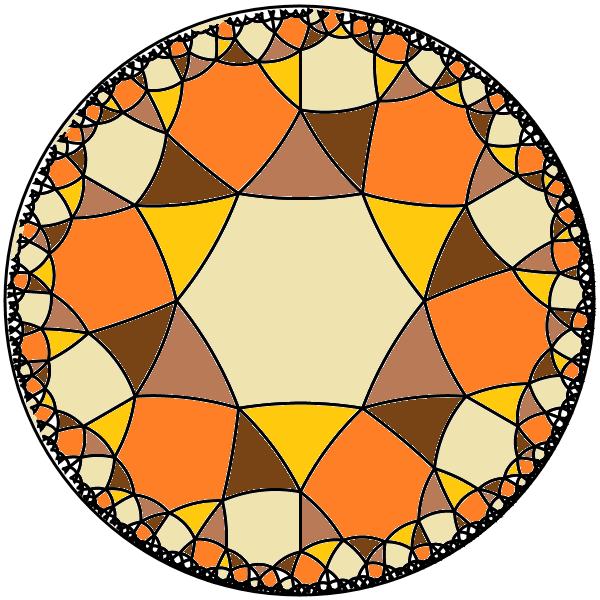}\\
					(3.3.6.3.6)_1&(3.3.6.3.6)_2
				\end{array}$
				\caption{Some chirally perfect precise 5-colorings of the tiling $(3.3.p.3.p)$.}  
				\label{33p3qexampleschiral}
			\end{figure}
	
		\subsection{6-valent Semiregular Planar Tilings}
	
			We finish this section with a discussion on the perfect precise 6-colorings of the family $\T=(3.p.3.q.3.q)$ of 6-valent semiregular tilings.
			
			If $p=q=3$ then $\T=(3^6)$ is the regular tessellation of the Euclidean plane by equilateral triangles.
			Its unique perfect precise $6$-coloring has been determined in~\cite{ri1}.
			If $p=q$ and $p>3$, then the symmetry group of $\T=(3.p.3.p.3.p)$ is $G={\ast}p33$.
			Every vertex of $\T$ is the center of a $3$-fold rotation.
			The stabilizer of every $p$-gon (triangle) is a dihedral group $D_p$ ($D_3$).
			Two orbits of tiles are formed under the action of $G$ on $\T$, 
			namely the set of $p$-gons and the set of triangles. 
			
			In Fig.~\ref{3p3p3phowto}, the six tiles meeting at vertex $v$ of $\T$ are assigned different colors: the $p$-gon $t$ is colored $1$,
			and the succeeding triangles and $p$-gons around $v$ are colored $2$, $3$, $4$, $5$, $6$, in that order, such that triangle $t'$ is colored $2$.
			Observe that $\T$ has a reflection symmetry $\alpha$ along the line that passes through the center of $p$-gon $t$ and vertex $v$.
			For the $6$-coloring of $\T$ to be precise, the $p$-gon~$a$ can only be assigned colors $1$, $2$, or $6$. 
			However, notice that $\alpha$ swaps colors $2$ and $6$.
			This means that $p$-gon~$a$ must get color~$1$ so that the resulting coloring of $\T$ will be perfect.
			
			\begin{figure}[htb]
				\centering				
				\includegraphics[height=1in]{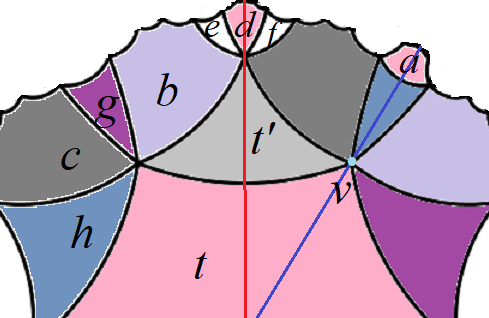}
				\caption{Obtaining a perfect precise $6$-coloring of the tiling $(3.p.3.p.3.p)$.}
				\label{3p3p3phowto}
			\end{figure}
			
			Let us denote by $\gamma$ the $3$-fold rotation about vertex $v$. 
			Since $\gamma$ is supposed to permute the colors in the coloring of $\T$,
			we conclude that $\gamma$ maps colors $1$, $3$, and $5$ to $3$, $5$, and $1$, respectively.
			This implies that the $p$-gon $b$ should be colored $5$. 
			
			We now consider the reflection symmetry $\beta$ of $\T$ along the line that passes through the centers of $p$-gon $t$ and triangle $t'$.
			Notice that $\beta$ fixes colors $1$ and $2$, and swaps $3$ and $5$.
			This implies that the $p$-gon $c$ should be colored $3$ and $p$-gon $d$ should get color $1$.
			Now, $\beta$ either fixes or swaps colors $4$ and $6$.
			However, triangles $e$ and $f$ are mapped to each other by the reflection~$\beta$.
			For the resulting $6$-coloring of $\T$ to be precise, we conclude that $\beta$ swaps colors $4$ and~$6$.
			Consequently, triangles $g$ and $h$ must be respectively assigned colors $6$ and $4$.
			This means that the $p$ triangles sharing an edge with $p$-gon $t$ should be colored $2$, $4$, and $6$ in a cyclical manner.
			The coloring of $\T$ is now uniquely determined and we see that $p$ must be a multiple of $3$ for a perfect precise $6$-coloring of $\T$ to be obtained.
			The unique perfect precise $6$-coloring of $\T$ generated corresponds to the partition $O_1\cup O_2$
			with $O_1=\set{g(J_1t_1)\mid g\in G}$, where $t_1$ is the $p$-gon stabilized by the $p$-fold rotation $QR$ and 
			$J_1=\gen{PRQP,Q,R}$ of index $3$ in $G$,
			and $O_2=\set{g(J_2t_2)\mid g\in G}$, where $t_2$ is the triangle stabilized by the $3$-fold rotation $PR$ and
			$J_2=\gen{P,QRPRPQ,R}$ of index $3$ in~$G$. 		
			Thus we obtain the following proposition.
			
			\begin{prop}\label{3p3p3presults}
				The semiregular tiling $(3.p.3.p.3.p)$ has exactly one perfect precise 6-coloring if and only if $p$ is a multiple of $3$.
			\end{prop}
			
			Fig.~\ref{3p3p3pexamples} shows the perfect precise $6$-colorings of the tilings $(3.6.3.6.3.6)$ and $(3.9.3.9.3.9)$.
			
			\begin{figure}[htb]
				\centering
				$\begin{array}{cc}
					\includegraphics[width=1.4in]{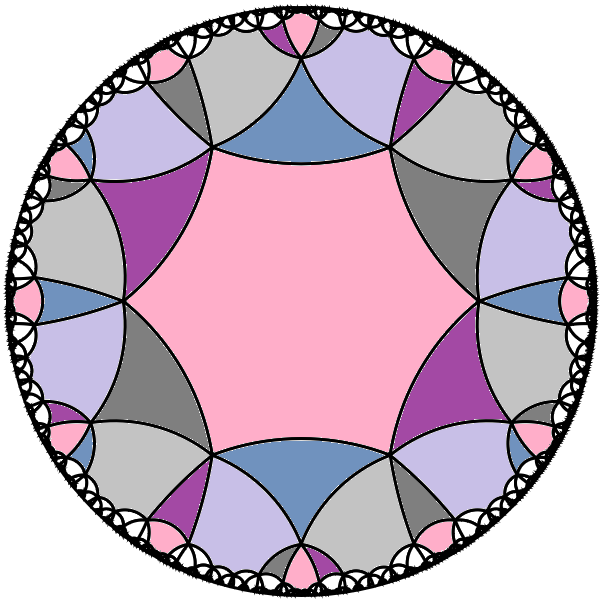}&
					\includegraphics[width=1.4in]{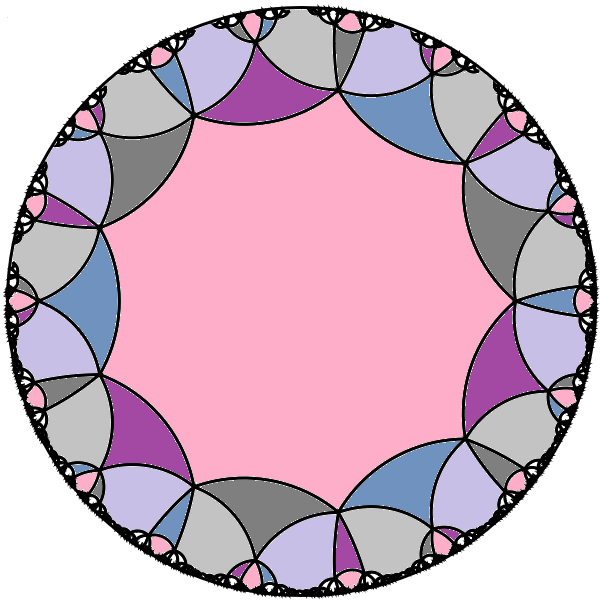}\\
					(3.6.3.6.3.6) & (3.9.3.9.3.9)\\
				\end{array}$
				\caption{Some perfect precise $6$-colorings of the tiling $(3.p.3.p.3.p)$.}
				\label{3p3p3pexamples}
			\end{figure}
			
			If $p\neq q$, then the symmetry group of $\mathcal{T}=(3.p.3.q.3.q)$ is $G=q{\ast}p$, 
			a subgroup of index $2$ of the triangle group ${\ast}(2p)q2$ (where $P$ still denotes the reflection opposite angle $\pi/(2p)$).
			Three orbits of tiles are formed under the action of $G$ on $\T$: 
			the set of $p$-gons, the set of $q$-gons, and the set of triangles.
			The stabilizer of every $p$-gon ($q$-gon) in $\T$ is a dihedral group $D_p$ (cyclic group $C_q$).
			We also note that no two $p$-gons are adjacent to each other, two $q$-gons are attached to each vertex of every $p$-gon, 
			and $q$ $q$-gons are attached to every $q$-gon at the vertices. 
			
			As earlier, we start with six tiles meeting at vertex $v$ of $\T$ with the following color assignment: the $p$-gon $t$ is colored $1$,
			and the succeeding triangles and $q$-gons around $v$ are colored $2$, $3$, $4$, $5$,~$6$, in that order, such that $q$-gon $t'$ is colored $3$ (see Fig.~\ref{3p3q3qhowto}).
			The tiling $\T$ still has the reflection symmetry $\alpha$ along the line that passes through the center of $p$-gon $t$ and vertex $v$.
			Using similar arguments, we conclude that the $p$-gon~$a$, and consequently every $p$-gon in $\T$, should be colored~$1$.
			
			\begin{figure}[htb]
				\centering
				\includegraphics[width=1.3in]{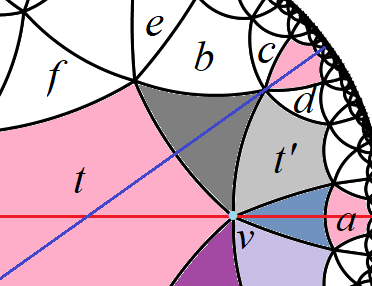}\\
				\caption{Obtaining a perfect precise 6-coloring of the tiling $(3.p.3.q.3.q)$, with $p\neq q$.}
				\label{3p3q3qhowto}
			\end{figure}
			
			For the resulting $6$-coloring of $\T$ to be precise, the $q$-gon $b$ should be assigned colors $4$, $5$, or $6$.
			We consider these cases separately, and we make use of the 
			reflection symmetry $\beta$ of $\T$ along the line that passes through the centers of $p$-gon $t$ and the triangle colored $2$,
			and the $q$-fold rotation symmetry $\gamma$ of $\T$ about the center of $q$-gon~$t'$.
			Notice that $\beta$ fixes colors $1$ and $2$, and swaps color~$3$ and the color of $q$-gon $b$, 
			while $\gamma$ fixes colors $1$ and $3$, and maps color $2$, $4$, $5$, and $6$ 
			to the color of triangle $d$, color $2$, the color of $q$-gon $b$, and the color of triangle~$c$, respectively.
			
			\begin{enumerate}[(i)]
				\item Suppose $q$-gon $b$ is colored $5$.
				This implies that the $q$-gons around $p$-gon $t$ are colored $3$ and $5$ alternately.
				Now, $\beta$ will either fix or swap colors $4$ and~$6$.
				However, triangles $c$ and $d$ should be colored $4$ and $6$ and are mapped to each other by the reflection $\beta$.
				This means that $\beta$ should swap colors $4$ and $6$, and so triangles $e$ and $f$ should be respectively colored $6$ and $4$.
				Hence, the $p$ triangles sharing an edge with $p$-gon $t$ will be colored $2$, $4$, and $6$ in a cyclical manner, which implies that $p$ is a multiple of $3$.
				
				Let us go back to triangles $c$ and $d$.
				If triangle $c$ is colored $6$ and triangle $d$ is colored~$4$,
				then the $q$-fold rotation $\gamma$ effects the color permutation $(2\,4)$.
				This implies that $q$ must be even (see Fig.~\ref{3p3q3qcases}(a)).
				On the other hand, if triangle $c$ is colored $4$ and triangle $d$ is colored~$6$,
				then $\gamma$ effects the color permutation $(2\,6\,4)$.
				Hence, we have $q$ to be a multiple of $3$ (see Fig.~\ref{3p3q3qcases}(b)).
				In both cases, we obtain a perfect precise $6$-coloring of $\T$ that corresponds to the partition $O_1\cup O_2\cup O_3$,
				with $O_1$ the set of $p$-gons, $O_2=\set{g(J_2t_2)\mid g\in G}$, 
				where $t_2$~is the $q$-gon stabilized by the $q$-fold rotation $PR$ and
				$J_2=\gen{PR,QRQP}$ of index $2$ in~$G$, %$G=\gen{Q,PR} \leq {\ast}(2p)q2$, 
				and $O_3=\set{g(J_3t_3)\mid g\in G}$, where $t_3$ is the triangle stabilized by the reflection $Q$.
				If $q$~is even then  $J_3=\gen{Q, (PR)^2, PR(QR)^2}$ of index $3$ in $G$,
				and if $q$ is a multiple of $3$ then $J_3=\gen{Q,PRPQR,RPRQR}$ of index $3$ in $G$.
				
				\item Suppose $q$-gon $b$ is colored $4$.
				Then $\beta$ interchanges colors $3$ and $4$ so triangle $e$ should be colored $3$,
				and interchanges colors $5$ and $6$ so triangle $f$ should be colored $5$.
				Therefore, the $p$~triangles sharing an edge with $p$-gon $t$ will be colored $2$, $5$, $4$, $3$, and $6$ in a cyclical manner.
				We obtain that $p$ must be a multiple of $5$.
				
				We again have two different ways of assigning colors to triangles $c$ and $d$.
				If triangle $c$ is colored $5$ and triangle $d$ is colored $6$,
				then the $q$-fold rotation $\gamma$ induces the permutation~$(2\,6\,5\,4)$ of colors. 
				Hence $q$ must be a multiple of $4$ (see Fig.~\ref{3p3q3qcases}(c)).
				If instead triangles~$c$ and $d$ are colored $6$ and $5$, respectively,
				then $\gamma$ induces the permutation $(2\,5\,4)$ of colors, and so $q$ is a multiple of~$3$ (see Fig.~\ref{3p3q3qcases}(d)).
				For both possibilities, a perfect precise $6$-coloring of~$\T$ is generated that corresponds to the partition $O_1\cup O_2$,
				with $O_1$ the set of $p$-gons and $O_2=\set{g(Jt_1\cup Jt_2)\mid g\in G}$,
				where $t_1$ is the $q$-gon stabilized by the $q$-fold rotation~$PR$ and 
				$t_2$ is the triangle stabilized by the reflection $(RQ)^3R$.
				If $q$ is a multiple of~$4$ then $J=\gen{PR,(RQ)^3R,QPR(PRQ)^2,Q(PR)^2Q(RP)^3Q}$ of index $5$ in $G$.
				On the other hand, if $q$~is a multiple of $3$ then $J=\gen{PR,(RQ)^3R,Q(RP)^3Q,(QR)^2QPRQ}$ of index $5$ in~$G$.
				
				\item Finally, suppose $q$-gon $b$ is colored $6$.
				Similar arguments show that the $p$ triangles sharing an edge with $p$-gon $t$ will be colored $2$, $3$, $4$, $5$, and $6$ in a cyclical manner, 
				and so $p$ is again a multiple of $5$.
				
				Assigning color $4$ and color $5$ respectively to triangles $c$ and $d$ will mean that the
				$q$-fold rotation $\gamma$ effects the color permutation $(2\,5\,6\,4)$.
				Hence $q$ must be a multiple of $4$ (see Fig.~\ref{3p3q3qcases}(e)).
				On the other hand, if triangle $c$ is colored $5$ and triangle $d$ is colored $4$, then
				$\gamma$ effects the color permutation $(2\,4)(5\,6)$ and so $q$ should be even (see Fig.~\ref{3p3q3qcases}(f)).
				In both cases, we obtain a perfect precise $6$-coloring of $\T$ that corresponds to the partition $O_1\cup O_2$,
				with $O_1$ the set of $p$-gons and $O_2=\set{g(Jt_1\cup Jt_2)\mid g\in G}$,
				where $t_1$ is the $q$-gon stabilized by the $q$-fold rotation~$PR$ 
				and $t_2$ is the triangle stabilized by the reflection $(RQ)^5R$.
				If $q$ is a multiple of $4$ then  $J=\gen{PR,(RQ)^5R,(PQR)^2Q,(QRPR)^2Q,Q(RP)^2(RQ)^2}$ of index $5$ in $G$.
				For the second possibility where $q$ is even, we have $J=\gen{PR,(RQ)^5R,(RQ)^2(PR)^2(QR)^2,(QR)^2P(RQ)^2P}$ of index $5$ in~$G$.			
			\end{enumerate}
		
			\begin{figure}[htb]
				\centering
				$\begin{array}{ccc}
					\includegraphics[height=1in]{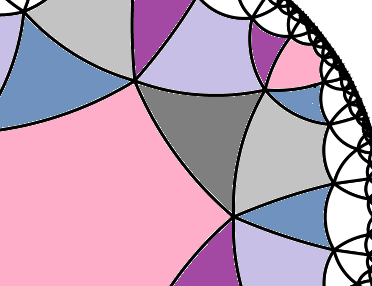}&
					\includegraphics[height=1in]{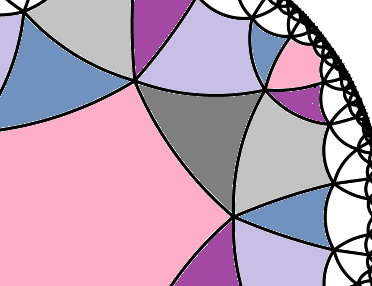}&
					\includegraphics[height=1in]{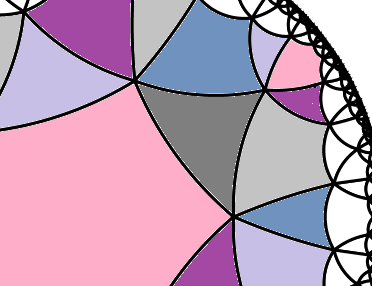}\\
					(a) & (b) & (c) \\
					\includegraphics[height=1in]{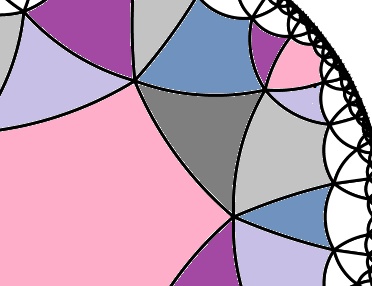}&
					\includegraphics[height=1in]{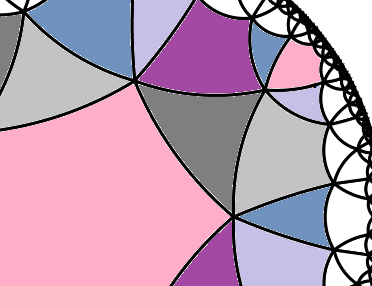}&
					\includegraphics[height=1in]{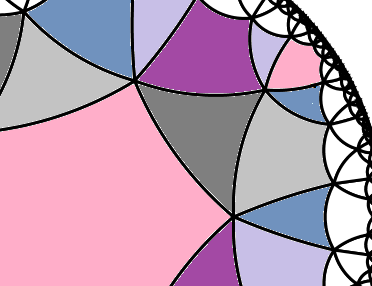}\\
					(d) & (e) & (f) \\
				\end{array}$
				\caption{Possible color assignments for a perfect precise 6-coloring of the tiling $(3.p.3.q.3.q)$, with $p\neq q$.}
				\label{3p3q3qcases}
			\end{figure}
			
			The following proposition summarizes the preceding discussion. 
			We illustrate some perfect precise $6$-colorings of the tiling $(3.p.3.q.3.q)$ with $p\neq q$ for some values of $p$ and $q$ in Fig.~\ref{3p3q3qexamples}. 
			
			\begin{figure}[htb]
				\centering
				$\begin{array}{ccc}
					\includegraphics[width=1.4in]{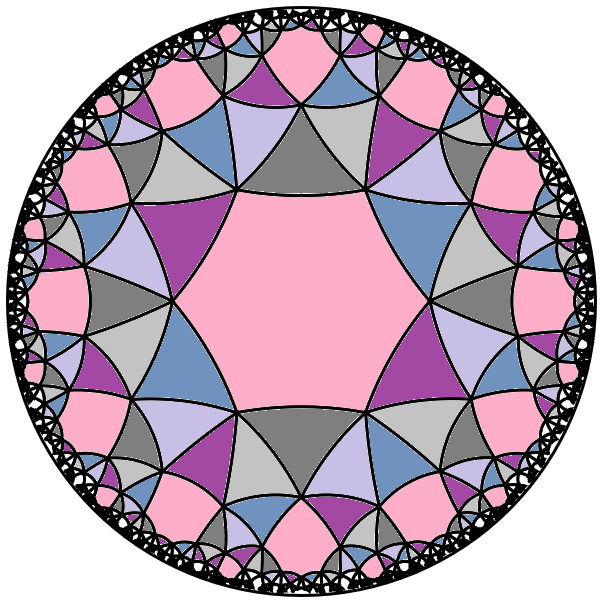}&
					\includegraphics[width=1.4in]{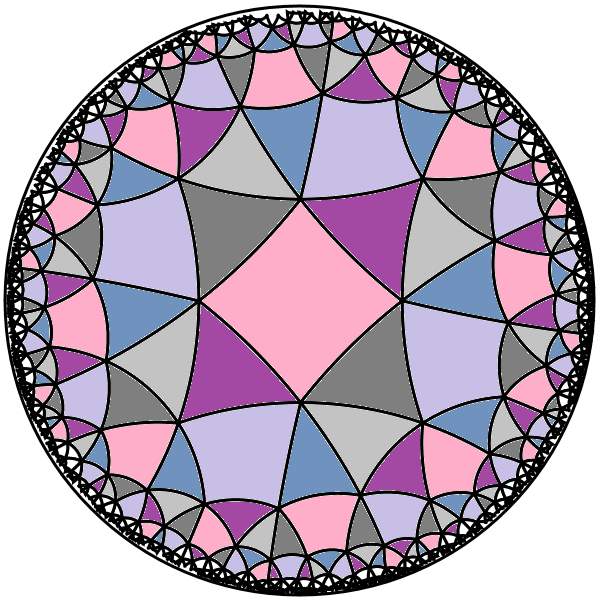}&
					\includegraphics[width=1.4in]{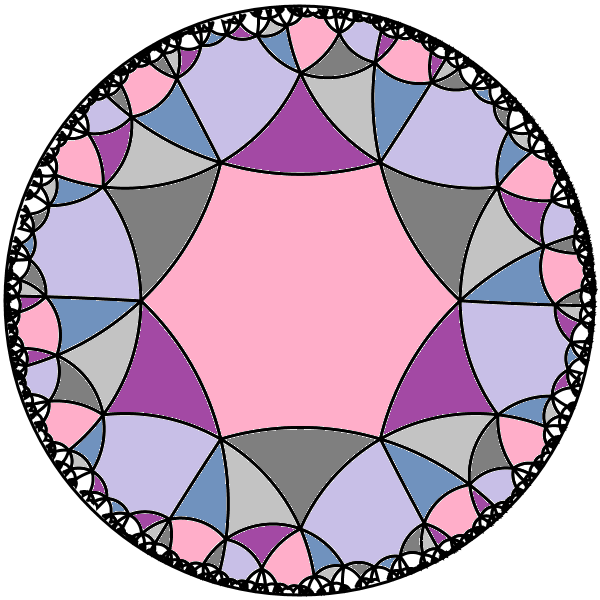}\\
					(3.6.3.3.3.3) & (3.3.3.4.3.4) &(3.3.3.6.3.6)_1 \\
					\includegraphics[width=1.4in]{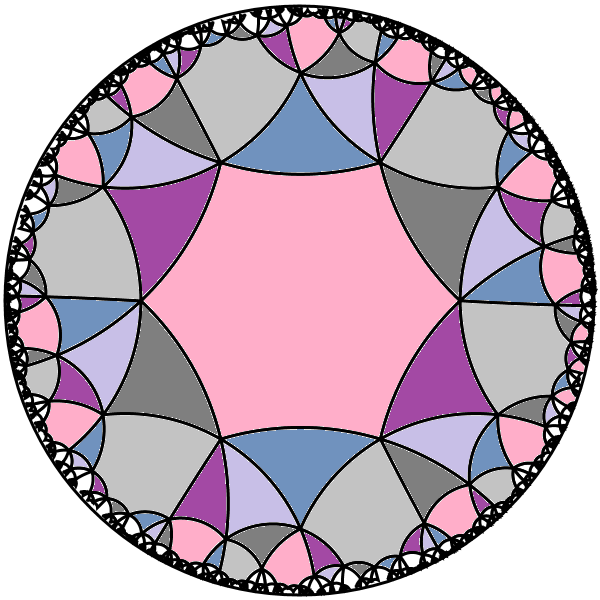}&
					\includegraphics[width=1.4in]{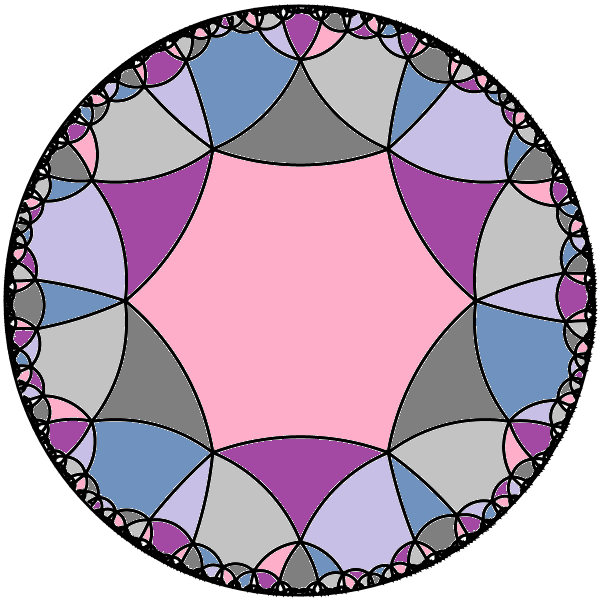}&
					\includegraphics[width=1.4in]{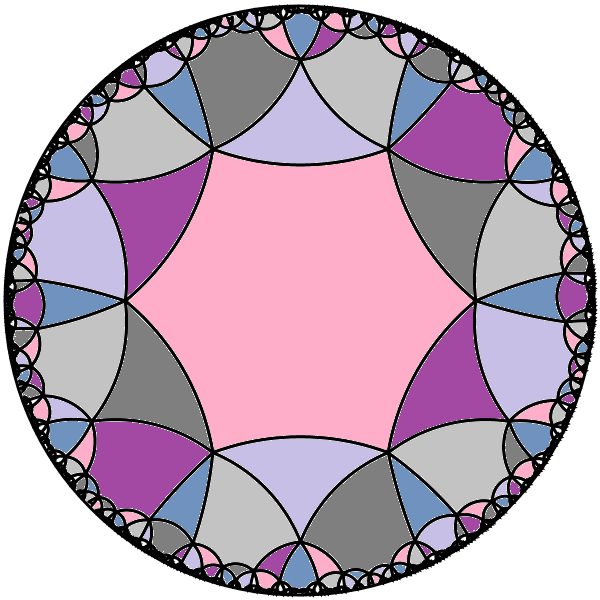}\\
					(3.3.3.6.3.6)_2&(3.5.3.6.3.6)_1  & (3.5.3.6.3.6)_2\\
					\includegraphics[width=1.4in]{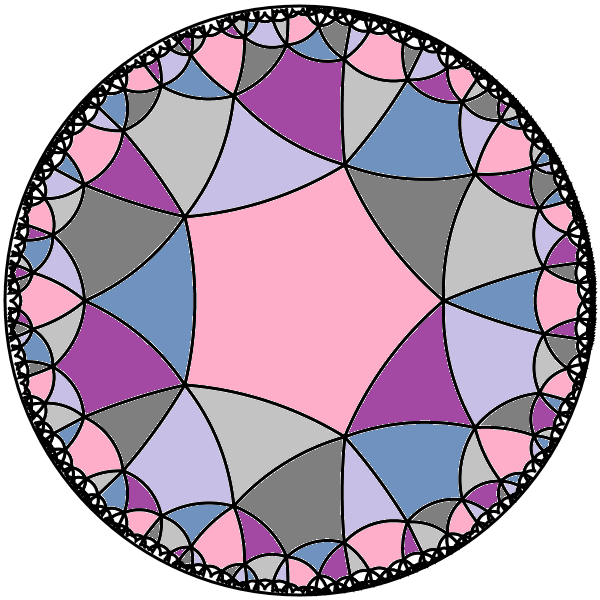}&
					\includegraphics[width=1.4in]{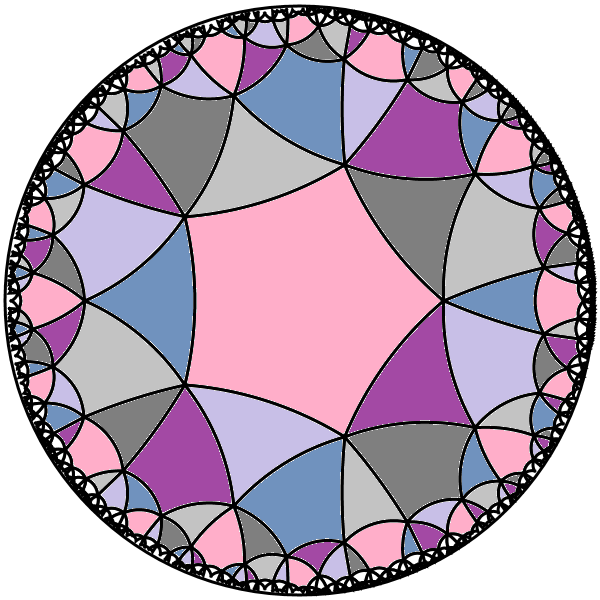}&
					\includegraphics[width=1.4in]{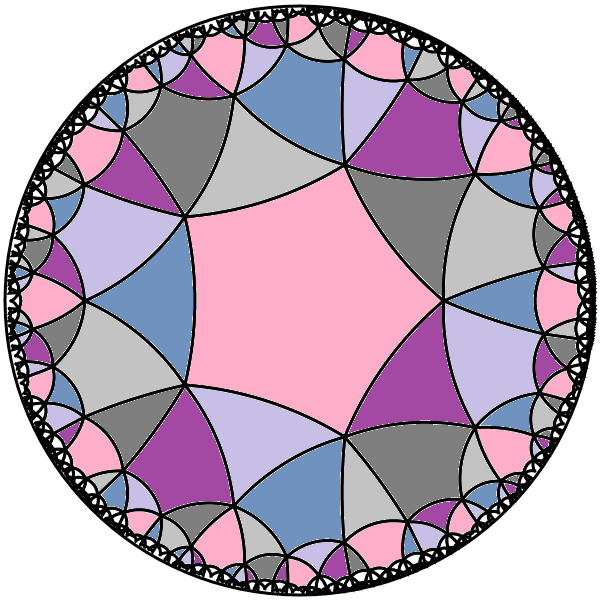}\\
					(3.5.3.4.3.4)_1 & (3.5.3.4.3.4)_2 & (3.5.3.4.3.4)_3
				\end{array}$
				\caption{Some perfect precise $6$-colorings of the tiling $(3.p.3.q.3.q)$, with $p\neq q$.}
				\label{3p3q3qexamples}
			\end{figure}
			
			\begin{prop}\label{3p3p3qresults}
				Let $\T$ be the semiregular tiling $(3.p.3.q.3.q)$ with $p\neq q$.
				Then the number of perfect precise $6$-colorings of $\T$ is
				\begin{enumerate}
					\item one if and only if one of the following hold: 
					$p$ is a multiple of $3$ but not of $5$ and $q$ is even,
					$p$ is a multiple of $5$ but not of $3$ and $q$ is even but not a multiple of $4$,
					$p$~is a multiple of $3$ but not of $5$ and $q$ is an odd multiple of $3$, or
					$p$ is a multiple of $5$ but not of $3$ and $q$~is an odd multiple of $3$;
					
					\item two if and only if one of the following hold:
					$p$ is a multiple of $15$ and $q$ is even but not a multiple of $4$,
					$p$ is a multiple of $15$ and $q$ is an odd multiple of $3$, 				
					$p$ is a multiple of $3$ but not of $5$ and $q$ is a multiple of $6$, or
					$p$ is a multiple of $5$ but not of $3$ and $q$ is a multiple of~$6$ but not of $4$;
					
					\item three if and only if $p$ is a multiple of $5$ but not of $3$ and $q$ is a multiple of $4$ but not of $3$;
					
					\item four if and only if one of the following hold:
					$p$ is a multiple of $15$ and $q$ is a multiple of $4$ but not of $3$,
					$p$ is a multiple of $15$ and $q$ is a multiple of $6$ but not of $4$, or
					$p$ is a multiple of $5$ but not of $3$ and $q$ is a multiple of $12$; and
					
					\item six if and only if $p$ is a multiple of $15$ and $q$ is a multiple of $12$.				
				\end{enumerate}
			\end{prop}
	
	\section{More Colors and Outlook}\label{finalsec}		
	
		In Section~\ref{methods}, we saw that Theorem~\ref{gencase} may be used to generate all (chirally) perfect precise colorings of a $k$-valent semiregular planar tiling $\T$.
		This process involves initially obtaining perfect colorings of a tiling $\T$ and subsequently filtering out those that are not precise. 
		An advantage of this procedure is that it allows us to obtain perfect colorings of semiregular planar tilings $\T$ where
		no two tiles sharing a common vertex have the same color and the number of colors this time is more than the valency of~$\T$. 
		For ease of reference, we shall still refer to such colorings as perfect precise colorings.
		
		To illustrate, we again consider the $(3.4.6.4)$ Euclidean tiling $\T$ with symmetry group $G$ of type~$\ast 632$.
		Let us obtain a perfect precise coloring of $\T$ with the least number of colors such that the three $G$-orbits of $\T$,
		namely, the orbit $X_1$ of hexagons, the orbit $X_2$ of squares, and the orbit of $X_3$ of triangles, do not share any colors. 
		That is, we want a partition of $\T$ of the form $\cup_{i=1}^3\{g(J_it_i)\mid g\in G\}$
		for some suitable tile $t_i\in X_i$ and subgroup $J_i$ of~$G$ such that $J_i$~contains~$\stab_G(t_i)$.
		Let us assign the same color to each tile in the orbit $X_1$, and also the same color to each tile in the orbit $X_3$.
		That is, we take $J_1=J_3=G$. 
		For the coloring to be precise, the orbit $X_2$ should be colored with at least three colors because of the three squares that surround each triangle in $\T$.  
		This means that we need to find a subgroup $J_2$ of at least index $3$ in $G$.
		We may take $J_2$ to be the subgroup $\gen{PRQRP,Q,R}$ of type $\ast 632$ to obtain a coloring of the orbit $X_2$ with three colors
		that correspond to the partition $\{g(J_2t_2)\mid g\in G\}$ where $t_2$ is the square stabilized by the two-fold rotation $PRQRPR$. 
		This will yield a perfect precise coloring of $\T$ with five colors.
		Another possibility is to take $J_2$ to be the subgroup $\gen{P,Q,QRPQRQ}$ of type $2{\ast}22$ where $t_2$ is the square stabilized by the two-fold rotation~$PQ$.
		The perfect precise $5$-colorings of $\T$ obtained are shown in Fig.~\ref{3464}.
		
		\begin{figure}[htb]
			\centering
			$\begin{array}{cc}
				\includegraphics[width=2in]{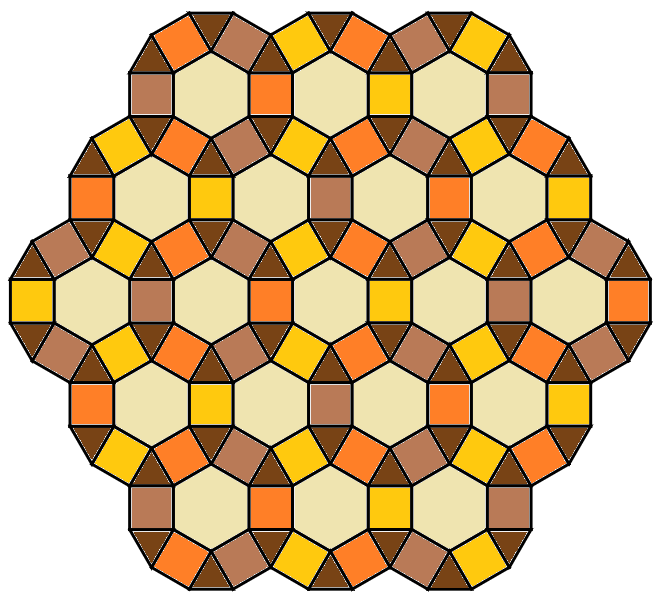} &
				\includegraphics[width=2in]{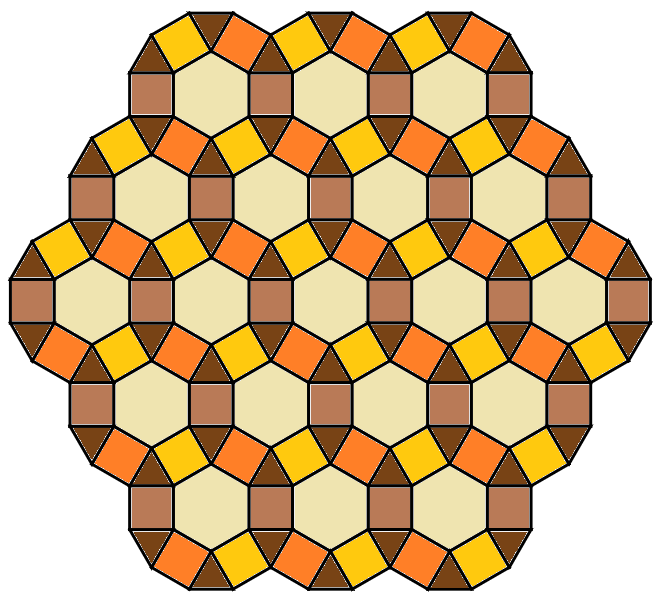}\\
				\text{subgroup }J_2 \text{ of type}\ast632 &
				\text{subgroup }J_2 \text{ of type }2*22
			\end{array}$
			\caption{Perfect precise 5-colorings of the tiling $(3.4.6.4)$.}
			\label{3464}
		\end{figure}
		
		We also provide examples of perfect precise colorings (with more than $4$ colors) of the $(6.4.6.4)$ hyperbolic tiling $\T$. 
		The orbit $X_1$ of hexagons and the orbit $X_2$ of squares are the two orbits of $\T$ formed under the action of its symmetry group $G$ of type $\ast 642$.
		Consider the following subgroups of $G$: $J_1=\gen{PRP,Q,R}$ of index $2$, $J_2=\gen{P,R,QRPRQ,QRQRQ}$ of index $3$, 
		${K_1=\gen{PQRQRP,Q,R}}$ of index $4$, and $K_2=\gen{P,R,QRQ}$ of index $2$.
		Take $t_1$ to be the hexagon stabilized by the six-fold rotation $QR$ and $t_2$ to be the square stabilized by the two-fold rotation $PR$.
		Then the $5$-coloring corresponding to the partition $\cup_{i=1}^2\set{g(J_it_i)\mid g\in G}$ of $\T$ 
		and the $6$-coloring corresponding to the partition $\cup_{i=1}^2\set{g(K_it_i)\mid g\in G}$ of $\T$ 
		are perfect and precise. 
		The colorings are shown in Fig.~\ref{6464more}.
		
		\begin{figure}[htb]
			\centering
			$\begin{array}{cc}
				\includegraphics[width=1.4in]{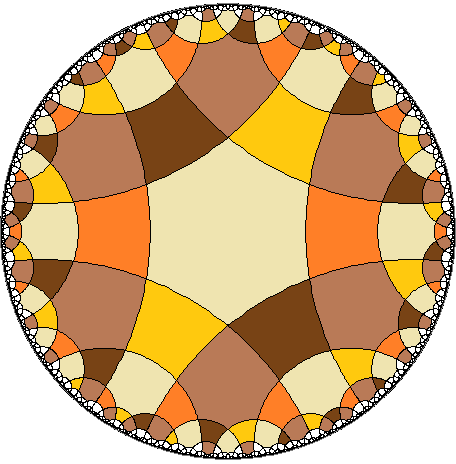} &
				\includegraphics[width=1.4in]{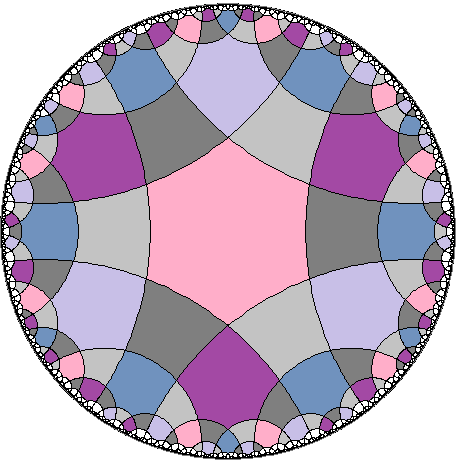}
			\end{array}$
			\caption{Perfect precise 5-coloring and $6$-coloring of the tiling $(6.4.6.4)$.}
			\label{6464more}
		\end{figure}
		
		In this contribution, we have presented two approaches to obtain perfect precise colorings of semiregular planar tilings.
		The first method involves a combinatorial analysis of the possible ways to color the tiling such that the resulting coloring is precise and perfect.
		By applying this method, we have identified all perfect precise colorings for certain families of $k$-valent semiregular planar tilings where $k\leq 6$.
		Alternatively, group-theoretic methods can also be used to obtain perfect precise colorings of tilings.
		This approach is applicable in finding the perfect precise coloring of a specific semiregular tiling. 
		It is also particularly useful when dealing with a broader definition of perfect precise colorings, wherein the number of colors used exceeds the valency of the tiling.
		
		Recall that in a semiregular planar tiling $\T$, the action of the symmetry group on the set of vertices is transitive. 
		As a result, when we have a perfect coloring of $\T$, 
		we only need to ensure that no two tiles sharing one vertex have the same color to confirm that the coloring is also precise.
		Building on this observation, it would be interesting to determine the perfect precise colorings of isogonal tilings, as well as $k$-isogonal tilings.
		In Fig.~\ref{isogonal}, we provide an example of a perfect precise coloring of the isogonal tiling IG 71.

		\begin{figure}[htb]
			\centering
			\includegraphics[width=2in]{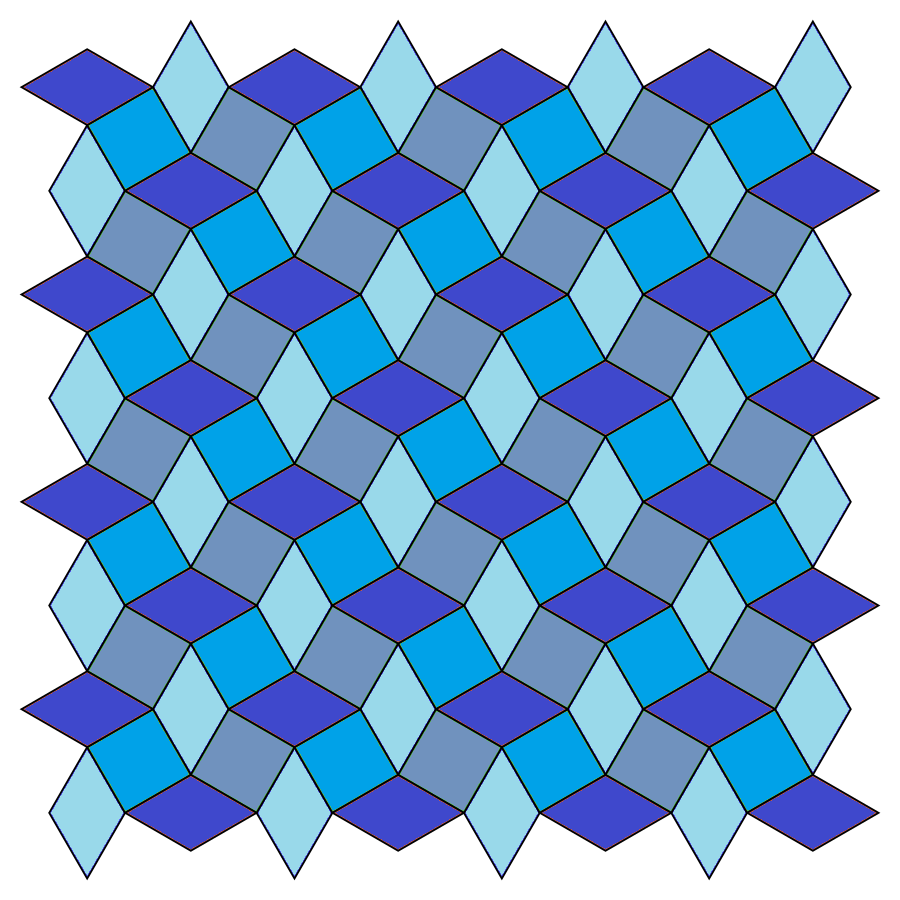}
			\caption{A perfect precise coloring of the isogonal tiling IG 71.}
			\label{isogonal}
		\end{figure}	

	\section*{Acknowledgements}
		This article is dedicated to the memory of our mentor, Professor Ren\'e P.~Felix. 
		The authors are grateful to the anonymous referees for their valuable remarks, which have helped improve the quality of the paper.

	\bibliographystyle{abbrv}
	\bibliography{mybib}
\end{document}